\documentclass[11pt]{article}
\usepackage{latexsym,amsmath,stackrel,color,amsthm,epsfig,array}
\usepackage{graphicx}
\usepackage{amssymb}
\usepackage[left=0.9in,top=0.9in,right=0.9in,bottom=0.9in]{geometry}
\usepackage[linktocpage=true]{hyperref}
\usepackage{setspace}
\usepackage{mathrsfs}
\usepackage{caption}
\usepackage{titlesec}
\usepackage{float}
\usepackage{comment,caption}
\usepackage{tikz}
\usepackage{thmtools}

\newcounter{foo}
\makeatletter
\def\thm@space@setup{%
  \thm@preskip=\parskip \thm@postskip=0pt
}
\makeatother

\makeatletter
\newcommand{\dotDelta}{{\vphantom{\triangle}\mathpalette\d@tD@lta\relax}}
\newcommand{\d@tD@lta}[2]{%
  \ooalign{\hidewidth$\m@th#1\mkern-1mu\cdot$\hidewidth\cr$\m@th#1\triangle$\cr}%
}
\makeatother

\makeatletter
\renewenvironment{proof}[1][\proofname]{\par
  \vspace{-\topsep}
  \pushQED{\qed}%
  \normalfont
  \topsep8pt \partopsep0pt 
  \trivlist
  \item[\hskip\labelsep
        \itshape
    #1\@addpunct{.}]\ignorespaces
}{%
  \popQED\endtrivlist\@endpefalse
  \addvspace{6pt plus 6pt} 
}
\makeatother

\makeatletter
\def\thm@space@setup{%
  \thm@preskip=0.3cm
  \thm@postskip=0cm
}
\makeatother

\declaretheoremstyle[%
  spaceabove=6pt,%
  spacebelow=6pt,%
  headfont=\normalfont\itshape,%
  postheadspace=1em,%
  qed=\qedsymbol%
]{mystyle}

\def\qed{\hfill\ifhmode\unskip\nobreak\fi\quad\ifmmode\Box\else\hfill$\Box$\fi}
\def\ite#1{\hfill\break${}$\hbox to 50pt {\quad(#1)\hfill}}
\newtheorem{thm}{Theorem}

\newtheorem{defn}{Construction}

\newtheorem{lemma}{Lemma}[section]
\newtheorem{conj}[foo]{Conjecture}

\newtheorem{claim}{Claim}[section]

\newtheorem{problem}{Problem}[section]

\def\ex{{\rm{ex}}}

\tikzstyle{vertex}=[circle,fill=black,inner sep=2pt]
\tikzstyle{vertrect}=[draw,rectangle,inner sep=2pt]
\tikzstyle{vertdia}=[draw,diamond,inner sep=2pt]

\newcommand{\TT}{\mathcal{T}}

\newcommand{\h}{H}

\newcommand{\tria}[1]{ \triangle ( #1) }
\newcommand{\lina}[1]{ \ell ( #1) }
\newcommand{\de}[1]{ \, {\mathrm{del}} \, ( #1) }

\newcommand{\cex}{\text{\rm ex}_\circlearrowright}

\titleformat{\subsection}[runin]
{\normalfont\bfseries}{\thesubsection}{1em}{}

\begin{document}

\pagestyle{myheadings}
\markright{{\small{\sc F\"uredi,  Mubayi, O'Neill, and Verstra\"ete: Extremal problems for pairs of triangles}}}

\title{Extremal problems for pairs of triangles}

\author{Zolt\'an F\"{u}redi\footnote{Alfr\'ed R\'enyi Institute of Mathematics,
Hungarian Academy of Sciences, P.O. Box 127, Budapest, Hungary, H-1364. Research was supported in part by
NKFIH grant KH130371 and NKFI--133819. E-mail: z-furedi@illinois.edu}
\and
	Dhruv Mubayi\footnote{Department of Mathematics, Statistics, and Computer Science, University of Illinois at Chicago.
Research supported by NSF award DMS-1952767. E-mail: mubayi@uic.edu}
	\and
	Jason O'Neill \footnote{Department of Mathematics, University of California, San Diego. Research supported by NSF award DMS-1800332. E-mail: jmoneill@ucsd.edu and jacques@ucsd.edu \newline\indent
}
 \and
	Jacques Verstra\"{e}te \footnotemark[3]
}

\maketitle

\begin{abstract}
A {\em convex geometric hypergraph} or cgh consists of a family of subsets of a strictly convex set of points in the plane.
There are eight pairwise nonisomorphic cgh's consisting of two disjoint triples.
These were studied at length by Bra{\ss}~\cite {Brass1} (2004) and by Aronov, Dujmovi\'c, Morin, Ooms, and da Silveira~\cite{ADMOS} (2019).
We determine the extremal functions exactly for seven of the eight configurations.

The above results are about cyclically ordered hypergraphs.
We extend some of them for triangle systems with vertices from a non-convex set.
We also solve problems posed by P.~Frankl, Holmsen and Kupavskii~\cite{FHK} (2020),
 in particular, we determine the exact maximum size of an intersecting family
of triangles whose vertices come from a set of $n$ points in the plane.
\end{abstract}

\section{Introduction}\label{sec_intro}

A {\em triangle system} is a pair $(P,\mathcal{T})$ where $P$ is a set of points in the plane in {\em general position}, i.e., no three collinear, and $\mathcal{T}$ is a set of  triangles with vertices from $P$. (A triangle is a closed set, the convex hull of three points not on a line). A {\em convex triangle system} is a triangle system $(P,\mathcal{T})$ where the elements of $P$ are in strictly convex position. It is convenient to treat $P$ in this case as the vertex set $\Omega_n$ of a regular $n$-gon in the plane, and to consider $\mathcal{T}$ to be a {\em convex geometric hypergraph} or {\em cgh} -- the vertex set is $\Omega_n$ with the clockwise cyclic ordering, and $\mathcal{T}$ is a set of triples from $\Omega_n$ called {\em edges} corresponding to the triples of vertices forming triangles. In this language, a cgh $\mathcal{S}$ is {\em contained} in a cgh $\mathcal{T}$ if there is an injection from the vertex set of $\mathcal{S}$ to the vertex set of $\mathcal{T}$ preserving the cyclic ordering of the vertices and preserving edges, and we say that a cgh $\h$ is {\em  $F$-free} if $\h$ does not contain $F$. In this paper, we concentrate on extremal problems for pairs of triangles in triangle systems and convex geometric hypergraphs. For the rich history of ordered and convex geometric graph problems and their applications, see \cite{CP,FKMV1,HP,KeP,Kup,KuP,PP} and the surveys of Pach~\cite{P1,P2} and Tardos~\cite{T}, and for convex triangle systems and generalizations, see \cite{Brass2,FJKMV2,PT} and the survey of Bra{\ss}~\cite{Brass1}.
On the other hand, the field of extremal hypergraph problems in the convex or geometric setting has fewer results, and statements of general principles in the area are lacking.
A natural first step in building such a theory is to solve interesting special cases, and this is one of the goals of this paper.

\subsection{Intersecting triangle systems}

An old theorem of Hopf and Pannwitz~\cite{HP} and Sutherland~\cite{S} states that the maximum number of line segments between $n$ points in the plane
with no two line segments disjoint is $n$. It is natural to ask for the maximum number of triangles between $n$ points in the plane with no two triangles
disjoint. To this end, a triangle system $(P,\mathcal{T})$ is {\em intersecting} if any two triangles in $\mathcal{T}$ share at least one point, and
{\em strongly intersecting} if any two triangles in $\mathcal{T}$ share a point in their interior. Intersecting triangle systems are motivated by the Erd\H{o}s-Ko-Rado Theorem~\cite{EKR}, and motivation for considering strong intersection is the
well-known theorem of Boros and the first author~\cite{ZF33} concerning the {\em depth} of points. They proved that for every set of $n$ points in the plane, the complete triangle system contains $\frac{2}{9}{n \choose 3}$ triangles with a common point in their interior (see also Bukh~\cite{Bukh}, Bukh, Matou\v{s}ek and Nivasch~\cite{BMN}, and B\'{a}r\'{a}ny~\cite{Barany}, Gromov~\cite{Gromov} and Karasev~\cite{Karasev} for the $d$-dimensional analogue). In particular, a strongly intersecting subfamily
of size at least $\frac{2}{9}{n \choose 3}$ exists. P.~Frankl, Holmsen and Kupavskii~\cite{FHK} recently determined that the maximum number of triangles in an $n$-point strongly intersecting convex triangle system is
$$\dotDelta(n)= \left\{\begin{array}{ll}
\displaystyle{\frac{n(n - 1)(n + 1)}{24}}  & \hbox{ if $n$ is odd}  \\[8pt]
\displaystyle{\frac{n(n - 2)(n + 2)}{24}} & \hbox{  if $n$ is even.}\end{array}\right.
$$
In particular, $\dotDelta(n)/{n \choose 3} \rightarrow 1/4$ as $n \rightarrow \infty$. The quantity $\dotDelta(n)$ also defines the maximum {\em depth} of a point in sets of $n$ points in the plane, which can be proved using the {\em upper bound theorem} for convex polytopes -- see Wagner and Welzl~\cite{WW}.
An $n$-point strongly intersecting convex triangle system of size $\dotDelta(n)$
is obtained by taking all triangles containing the centroid of $\Omega_n$ when $n$ is odd, together with all triangles on one side of each diameter of
$\Omega_n$ when $n$ is even (these constructions have size $\dotDelta(n)$, see~\cite{ZF1} for instance). For convenience, we let $\mathcal{H}^{\star}(n)$ denote the family of all such convex triangle systems with $n$ points.  P.~Frankl, Holmsen and Kupavskii posed the following problem (see Problem~1 in~\cite{FHK}):

\begin{problem} \label{maxd1}
What is the maximum size, over all point sets of size $n$, of the largest strongly intersecting triangle system? Is the maximum
always at most $\left(\frac{1}{4} + o(1)\right){n \choose 3}$ as $n \rightarrow \infty?$
\end{problem}

Our first result solves this problem completely for point sets in general position, as follows:

\begin{thm} \label{thm:d1}
Any $n$-point strongly intersecting triangle system has size at most $\dotDelta(n)$.
\end{thm}

The short proof of Theorem \ref{thm:d1} is given in Section \ref{sec:proofofthmd1}.
Note that Theorem~\ref{thm:d1} sharpens and extends the main result of~\cite{FHK} cited above, as $\dotDelta(n)$ is exactly the size of
every convex triangle system in $\mathcal{H}^{\star}(n)$. P.~Frankl, Holmsen and Kupavskii further posed the problem of determining the maximum size of an $n$-point intersecting convex triangle system if one allows triangles to intersect on the boundary (see Problem 2 in~\cite{FHK}):

\begin{problem}\label{problem}
What happens if one relaxes the intersecting condition and allows triangles to intersect on the boundary?
\end{problem}

There are a number of different intersection patterns of pairs of triangles in convex triangle systems, depicted below:

\begin{figure}[H]
\centering
\includegraphics[scale=0.6]{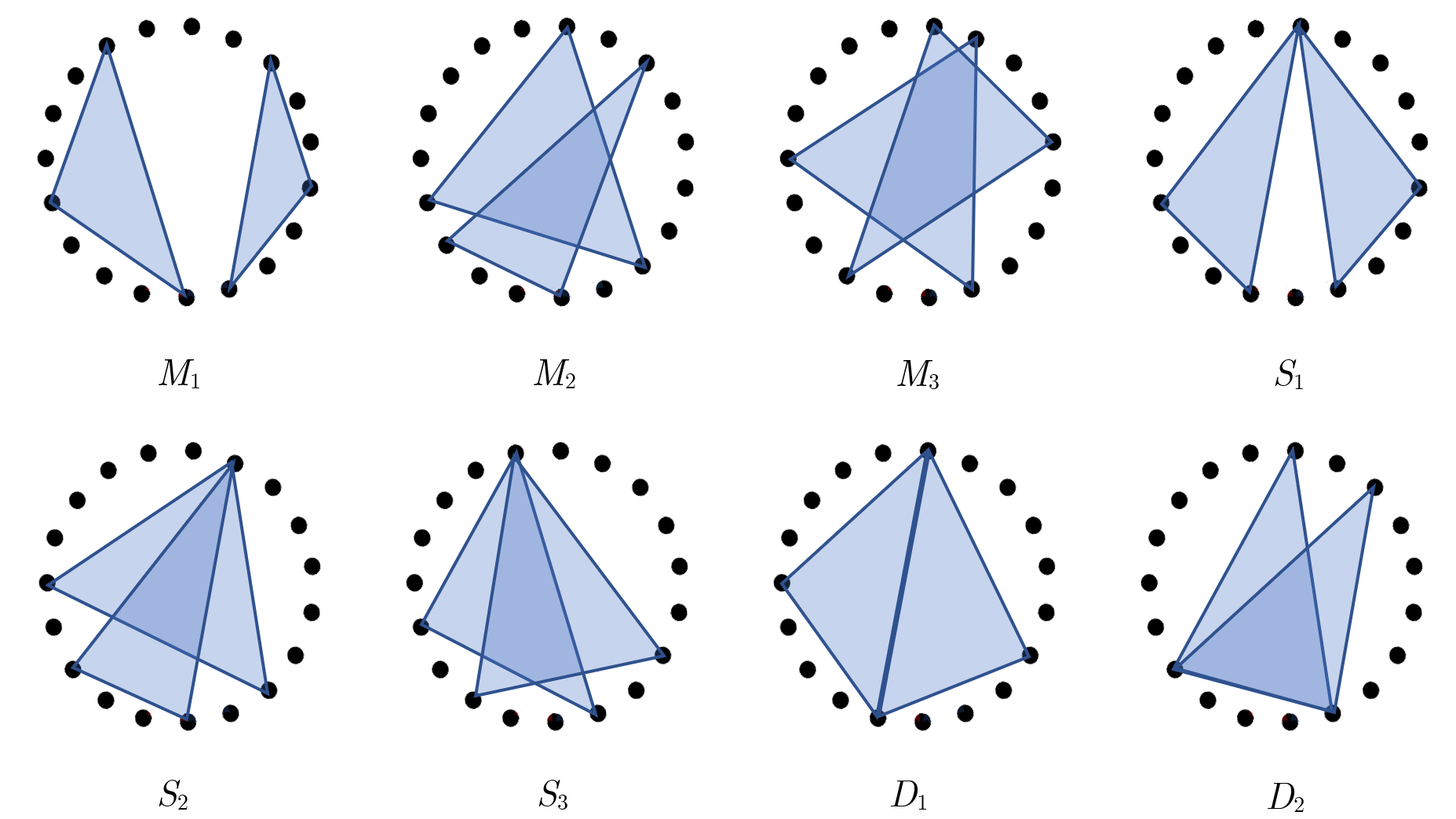}
\caption{The eight types of triangle pairs in convex triangle systems}
\label{allconfig}
\end{figure}

For all of these configurations, Bra{\ss}~\cite{Brass1} has shown the extremal function for convex triangle systems is either $\Theta(n^2)$ or $\Theta(n^3)$; the latter arises precisely when the two triangles have no common interior point. Aronov, Dujmovi\'c, Morin, Ooms and da Silveira~\cite{ADMOS} extensively studied cghs which avoid combinations of the configurations in Figure \ref{allconfig}, and determined many of the order of magnitudes of the associated extremal numbers.
An intersecting convex triangle system is precisely a convex triangle system not containing $M_1$, and a strongly intersecting convex triangle system is precisely a convex triangle system containing none of $M_1,D_1$ and $S_1$. If ${\cal F}$ is a set of convex triangle systems, then we denote by $\cex(n,{\cal F})$ the maximum size of a convex triangle system not containing any member of ${\cal F}$. In this language, P.~Frankl, Holmsen and Kupavskii~\cite{FHK} proved
$\cex(n,\{D_1,M_1,S_1\}) = \dotDelta(n)$. Problem \ref{problem} asks for $\cex(n,{\cal F})$ where ${\cal F} \subseteq \{M_1,D_1,S_1\}$ and we completely solve this problem using the following theorem:

\begin{thm}\label{thm:all}
For all $n \ge 3$,
\begin{eqnarray*}
\cex(n,F) &=& \left\{\begin{array}{ll}
\dotDelta(n) & \mbox{ if }F = D_1 \\
\dotDelta(n) + \lfloor\frac{n}{2}\rfloor
\lfloor\frac{n-2}{2}\rfloor & \mbox{ if }F = S_1 \\
\dotDelta(n) + \frac{n(n-3)}{2} & \mbox{ if }F = M_1. \end{array}\right.
\end{eqnarray*} %
\end{thm}

Furthermore, the extremal constructions for this theorem are classified -- see the constructions in Section \ref{sec:constructions}.  Using Theorem \ref{thm:all}, we obtain the exact value of $\cex(n,{\cal F})$ for each ${\cal F} \subseteq \{M_1,S_1,D_1\}$:
$$\cex(n, {\cal F}) = \min_{F \in {\cal F}}  \, \cex(n, F).$$
The extremal constructions above are characterized in our proofs in all cases except ${\cal F}=\{D_1, M_1\}$.

We also answer Problem~\ref{problem} in the more general context of triangle systems.
In this setting, $D_1$ denotes two triangles on opposite sides of a line and sharing a side -- tangent triangles -- and
$S_1$ denotes two triangles intersecting in exactly one vertex -- touching triangles -- whereas $M_1$ denotes two triangles sharing no points -- separated triangles.
Theorem \ref{thm:d1} as well as the first two parts of Theorem \ref{thm:all} are an immediate consequence of the following stronger theorem:

 \begin{thm}\label{thm:allplanar}
 Let $F \in \{M_1,D_1,S_1\}$, and let $\mathcal{T}$ be an $n$-point triangle system of maximum size not containing $F$. Then
 \[ |\mathcal{T}| = \left\{\begin{array}{ll}
\dotDelta(n) & \mbox{ if } F = D_1 \\
\dotDelta(n) + \lfloor\frac{n}{2}\rfloor
\lfloor\frac{n-2}{2}\rfloor & \mbox{ if }F = S_1 \\
\dotDelta(n) + \Theta(n^2) & \mbox{ if }F = M_1.\end{array}\right.\]
\end{thm}

The additive term of order $n^2$ for the case of $M_1$ in Theorem \ref{thm:allplanar} arises from a geometric theorem of Valtr on {\em avoiding line segments} in the plane.
We believe that the value of $\cex(n,M_1)$ should determine the maximum for $n$-point intersecting triangle systems:

\begin{conj}\label{conj:m1}
For all $n \geq 3$, if $\mathcal{T}$ is an $n$-point intersecting triangle system, then $|\mathcal{T}| \leq \cex(n,M_1)$.
\end{conj}

For the above configurations $F\in \{ D_1, S_1, M_1\}$,
 the extremal functions $\cex(n,F)$ and for the planer triangle systems, $\ex(P,F)$, are equal
(almost equal).
This is quite exceptional, for most configurations $F$ the non-convex case is much more complex.
E.g., one can find a self-intersecting path $P_3$ of length three
  in a convex geometric graph with $\Omega(n)$ edges, while for the general
not necessarily convex case Pach, Pinchasi, Tardos, and T\'oth~\cite{PPTT} showed that $\max \ex(P_3,F)=\Omega( n \log n)$.

\subsection{The five configurations in the $\Theta(n^2)$ range}~\label{sec:results5}

Bra{\ss}~\cite{Brass1} has shown that the five configurations whose extremal function is in the $\Theta(n^2)$ range are $S_2,S_3,M_2,M_3$ and $D_2$.
 In this section, we determine $\cex(n,F)$ exactly for $F \in \{S_3,M_2,M_3\}$ and give bounds for $F \in \{S_2,D_2\}$. The extremal function for $M_3$ was determined exactly in \cite{FJKMV1}. We also determine the exact extremal function for $M_2$ and $S_3$ when $n$ is even:

\begin{thm}\label{thm:m3m2s3}
\[
 \cex(n,F) = \left\{\begin{array}{ll}
 {n \choose 3} - {n - 3 \choose 3} & \mbox{ if }F = M_3 \mbox{ and }n \geq 3. \\
 {n \choose 2} - 2 & \mbox{ if }F = M_2 \mbox{ and }n \geq 7.\\
 \frac{n(n-2)}{2} & \mbox{ if }F = S_3\mbox{ and }n \geq 4\mbox{ is even}.
 \end{array}\right.\]
\end{thm}

For $S_3$ when $n$ is odd,  there are several constructions which obtain the lower bound $\cex(n, S_3)\ge \frac{(n-1)(n-2)}{2}+1$ (see Construction~\ref{s3free} in Section~\ref{sec:constructions}), but we have not proved that this bound is sharp. We leave the following open problem:

\begin{problem}\label{prob:s3}
Prove $\cex(n,S_3) = (n - 1)(n - 2)/2 + 1$ when $n\geq 5$ is odd, and characterize the extremal $S_3$-free convex geometric hypergraphs.
\end{problem}

The configurations  $S_2$ and $D_2$ appear to be the most difficult to handle. 

\begin{thm}\label{thm:configs2}
For $n \geq 3$,
\[ \Bigl\lfloor\frac{n^2}{4}\Big\rfloor - 1 \le \cex(n,S_2) \leq \frac{23}{64}n^2.\]
\end{thm}
We believe that the lower bound in this theorem is tight.

\begin{conj}\label{conj:config5} For all $n \geq 5$, $\cex(n,S_2) = \lfloor n^2/4 \rfloor - 1$.
\end{conj}

\begin{thm}\label{thm:configd2}
For $n\geq 3$,
\[ \frac{4}{9} {n \choose 2} - O(n) \leq \cex(n,D_2) \leq \frac{2n^2-3n}{9}.\]
\end{thm}

The lower bound is due to Dam\'asdi and N.~Frankl~\cite{DF} who solved our conjecture from an earlier draft of this paper and determined  $\lim_{n \rightarrow \infty} \cex(n,D_2)/{n \choose 2}$.
Even more, they showed that equality holds for all $n \equiv 6 \mod 9$ and gave an independent proof for our upper bound. 
Beside the upper bound  we present a lower bound
$\frac{3}{7} {n \choose 2} - O(n)\leq \cex(n,D_2)$ in Construction~\ref{d2free} using a quite different method.

\subsection{Summary of results.} We summarise the results for $\cex(n,F)$ in this paper in the following table. For $S_2$ and $D_2$,
we only have bounds on the extremal function, and write $[a,b]$ in the table to denote $a \leq \cex(n,F) \leq b$. We conjecture $\cex(n,S_2) = \lfloor n^2/4\rfloor  - 1$.
The constructions refer to those numbered 1 -- 8 in Section \ref{sec:constructions}.

\begin{center}
\begin{tabular}[c]{|c|c|c||c|c|c|}
\hline
& & & &  &    \\
$F$ &  $\cex(n,F)$ & {\small Construction} & $F$ & Bounds on $\cex(n,F)$ &  {\small Construction} \\ & & &  & & \\ \hline
\hline
\includegraphics[width= 20mm, height= 20mm,trim = {0 -0.1cm 0 -0.1cm}]{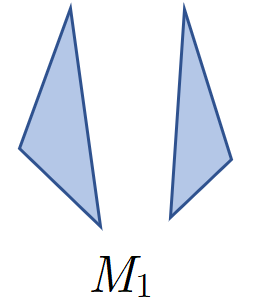}
& \raisebox{0.3in}{$\dotDelta(n)+\frac{n(n-3)}{2}$} & \raisebox{0.3in}{{\small 3}}
&  \includegraphics[width= 20mm, height=20mm,trim = {0 -0.1cm 0 -0.1cm}]{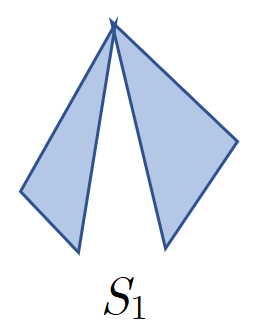}
& \raisebox{0.3in}{$\dotDelta(n) + \lfloor \frac{n}{2} \rfloor \lfloor \frac{n-2}{2}\rfloor$} &   \raisebox{0.3in}{{\small 2}} \\
\hline
\hline
\includegraphics[width= 20mm, height=20mm,trim = {0 -0.1cm 0 -0.1cm}]{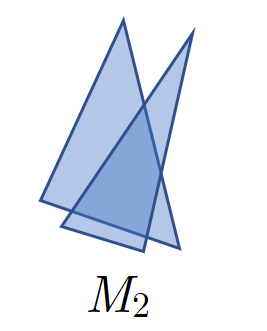}
& \raisebox{0.3in}{$\binom{n}{2}-2$} & \raisebox{0.3in}{{\small 5}}
&  \includegraphics[width= 20mm, height= 20mm,trim = {0 -0.1cm 0 -0.1cm} ]{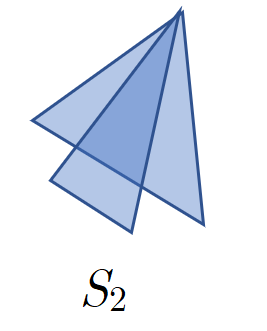}
& \raisebox{0.3in}{$[\lfloor \frac{n^2}{4} \rfloor -1 ,\frac{23n^2}{64}]$}  & \raisebox{0.3in}{{\small 7}} \\
\hline
\hline
\includegraphics[width= 20mm, height=20mm,trim = {0 -0.1cm 0 -0.1cm}]{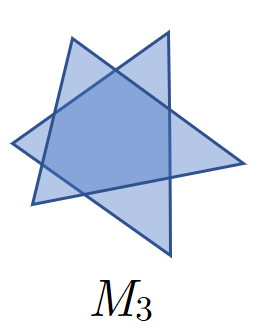}
& \raisebox{0.3in}{$\binom{n}{3} - \binom{n-3}{3}$} & \raisebox{0.3in}{{\small 4}}
&  \includegraphics[width= 20mm, height= 20mm,trim = {0 -0.1cm 0 -0.1cm}]{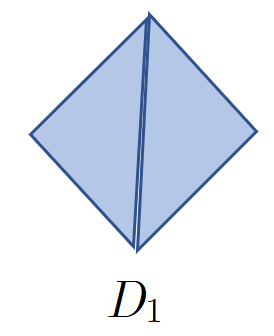}
& \raisebox{0.3in}{$\dotDelta(n)$} &  \raisebox{0.3in}{{\small 1}} \\
\hline
\hline
\includegraphics[width= 20mm, height=20mm,trim = {0 -0.1cm 0 -0.1cm}]{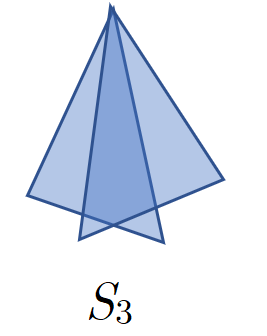}
& \raisebox{0.3in}{$\displaystyle{{\frac{n(n-2)}{2}}\atop{\mbox{for }n\mbox{ even}}}$} & \raisebox{0.3in}{{\small 6}}
&  \includegraphics[width= 20mm, height=20mm,trim = {0 -0.1cm 0 -0.1cm}]{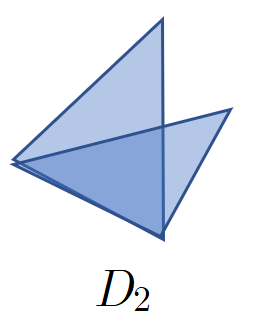}
& \raisebox{0.3in}{
                     $\displaystyle{\frac{2n^2-3n}{9}
                   }\atop{\mbox{for }n \equiv 6 \mod 9}  $} &  \raisebox{0.3in}{{\small 8}} \\
\hline
\end{tabular}
\end{center}

\subsection{Organization.}

Constructions of $F$-free convex triangle systems which give lower bounds for the theorems in this paper are in Section \ref{sec:constructions}, Constructions 1 -- 8. Sections~\ref{sec:proofofthmd1}, \ref{sec:proofofthms1}, \ref{sec:proofofthmm1} contain the proofs of our results for $D_1$, $S_1$ and $M_1$, respectively (i.e. the proofs of Theorems~\ref{thm:all} and~\ref{thm:allplanar}).
Section~\ref{sec:proofofthmm2m3} -- \ref{sec:configd2} contain the proofs of our results concerning the configurations $S_2,S_3,M_2,M_3,D_2$. Concluding remarks and further questions are in Section~\ref{sec:remarks}.

\subsection{Notation.}

We refer to a set of triangles from a set $\Omega_n$ of $n$ vertices of a regular $n$-gon as a {\em convex triangle system}. It is convenient also
to refer to this as a {\em convex geometric hypergraph} or {\em cgh}, where the triangles are considered as triples in ${\Omega_n \choose 3}$,
and the vertices of $\Omega_n$ are cylically ordered in the clockwise direction, say $v_0 < v_1 < \dots < v_{n - 1} < v_0$. In this case,
we consider the subscripts modulo $n$. A cgh $F$ is {\em contained} in a cgh $H$ if there is an injection from $V(F)$ to $V(H)$ preserving the cyclic ordering of the vertices and preserving edges, and we say that $\h$ is {\em $F$-free} if $\h$ does not contain $F$ as a subhypergraph. The extremal function $\cex(n,F)$ denotes the maximum number of edges in an $F$-free cgh on $\Omega_n$. Given $\h \subset \binom{\Omega_n}{3}$ and $A \subseteq \Omega_n$, let $d_\h(A) = |\{e \in H : A \subset e\}|$ be the {\em degree} of $A$ in $\h$; we write $d_H(u,v)$ when $A = \{u,v\}$ and $d_H(v)$ when $A = \{v\}$.
Let $\partial H = \{\{u,v\} : \exists e \in H, \{u,v\} \subset e\}$ denote the {\em shadow} of $H$. For functions $f,g :\mathbb N\rightarrow \mathbb R^+$, we write
 $f = o(g)$ if $\lim_{n \rightarrow \infty} f(n)/g(n) = 0$, and $f = O(g)$ if there is $c > 0$ such that $f(n)\leq cg(n)$ for all $n \in \mathbb N$. If $f = O(g)$ and $g = O(f)$, we write $f =\Theta(g)$.

\section{Constructions}\label{sec:constructions}

\begin{defn}[$D_1, S_1$ and $M_1$-free cghs]\label{1}
\normalfont For $n \geq 3$ odd, let the class of cghs $\mathcal{H}^{\star}(n)$ comprise the single cgh  consisting of triangles which contain in their interior the centroid of $\Omega_n$. For $n \geq 4$ even, each $H \in \mathcal{H}^\star(n)$ consists of all triangles which contain the centroid of $\Omega_n$ and, for each diameter $\{v_i,v_{i+n/2}\}$ of $\Omega_n$, we either add all triangles $\{v_i, v_j, v_{i + n/2}\}$ where $v_i < v_j < v_{i + n/2}$, or all triangles $\{v_i,v_j,v_{i+n/2}\}$ where $v_{i + n/2} < v_j < v_i$. It is not hard to show that each element $H \in \mathcal{H}^{\star}(n)$ has size $\dotDelta(n)$ -- see~\cite{ZF1}. 
Each $H \in \mathcal{H}^{\star}(n)$ is strongly intersecting, so
\begin{equation}\label{eqD1}
\cex(n, D_1) \ge \cex(n, \{D_1, S_1, M_1\}) \ge \dotDelta(n).
\end{equation}
\end{defn}

\begin{defn}[$S_1$ and $M_1$-free cghs]\label{3} \normalfont
 For $n \geq 3$ odd, each cgh in $\mathcal{H}^+(n)$ is obtained by adding for some $i < n$ to any cgh in $\mathcal{H}^{\star}(n)$ all triangles containing a pair $\{v_{i + j},v_{i + j + (n - 1)/2}\}$ for $0 \leq j \leq (n - 3)/2$ (left diagram in Figure \ref{def3}).  For $n \geq 4$ even, each $H \in \mathcal{H}^+(n)$ consists of all
triangles containing the centroid of $\Omega_n$ in their interior or on their boundary (right diagram in Figure \ref{def3}).

\begin{figure}[H]
\centering
\includegraphics[scale=0.4]{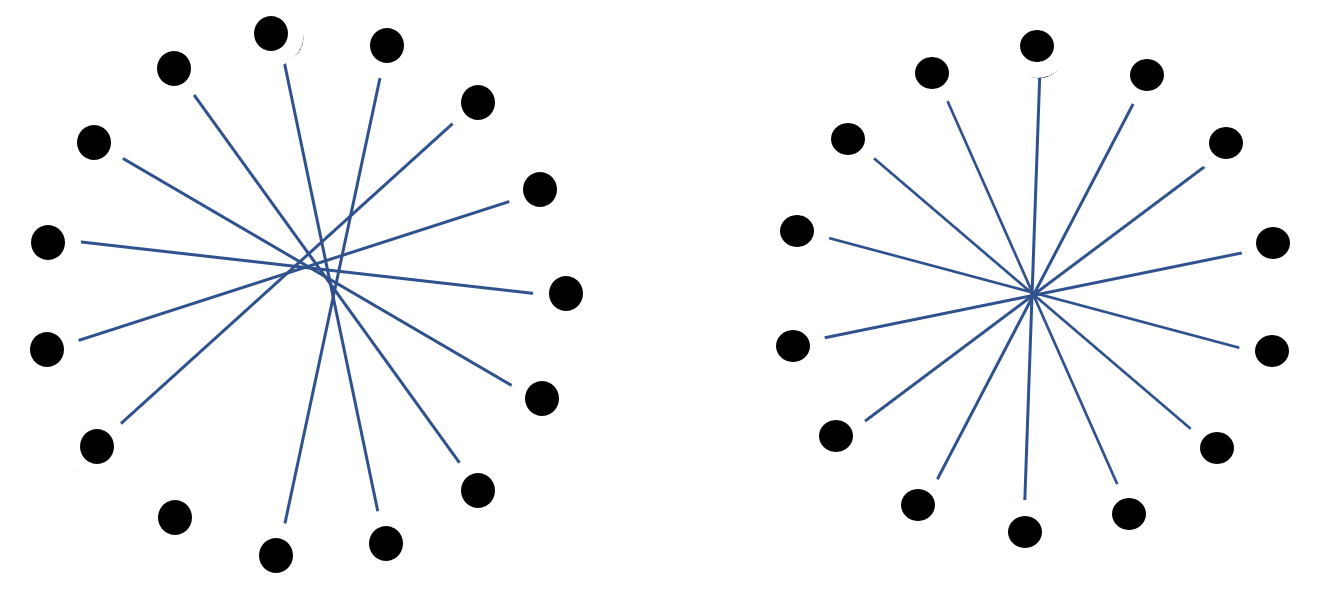}
\caption{Construction of $\mathcal{H}^+(n)$}
\label{def3}
\end{figure}

By inspection, each $\h \in \mathcal{H}^+(n)$  is $S_1$-free and $M_1$-free.
Moreover, if $n$ is odd, then $|H| = \dotDelta(n) + (n - 1)(n - 3)/4$  whereas if $n$ is even, then $|H|=\dotDelta(n) + n(n - 2)/4$.
 We obtain
\begin{equation}\label{eqS1}
\cex(n, S_1) \ge \cex(n, \{S_1, M_1\}) \ge \dotDelta(n) + \lfloor\tfrac{n}{2}\rfloor
\lfloor\tfrac{n-2}{2}\rfloor.
  \end{equation}
\end{defn}

\begin{defn}[$M_1$-free cghs]\label{2} \normalfont
For $n \geq 3$ odd, the unique cgh in $\mathcal{H}^{++}(n)$ is obtained by adding all triangles containing a pair $\{v_i,v_{i + (n - 1)/2}\}$ to the cgh in $\mathcal{H}^\star(n)$ (left diagram in Figure \ref{def2}). For $n \geq 4$ even, $\mathcal{H}^{++}(n)$ is obtained by adding all triangles containing a diameter of $\Omega_n$, plus all triangles
containing a pair from a set of $n/2$ pairwise intersecting pairs of the form $\{v_{i},v_{i+n/2-1}\}$ to any cgh in $\mathcal{H}^\star(n)$  (right diagram in Figure \ref{def2}). Every cgh in $\mathcal{H}^{++}(n)$ is $M_1$-free, and has size $\dotDelta(n) + n(n - 3)/2$.

\begin{figure}[H]
\centering
\includegraphics[scale=0.4]{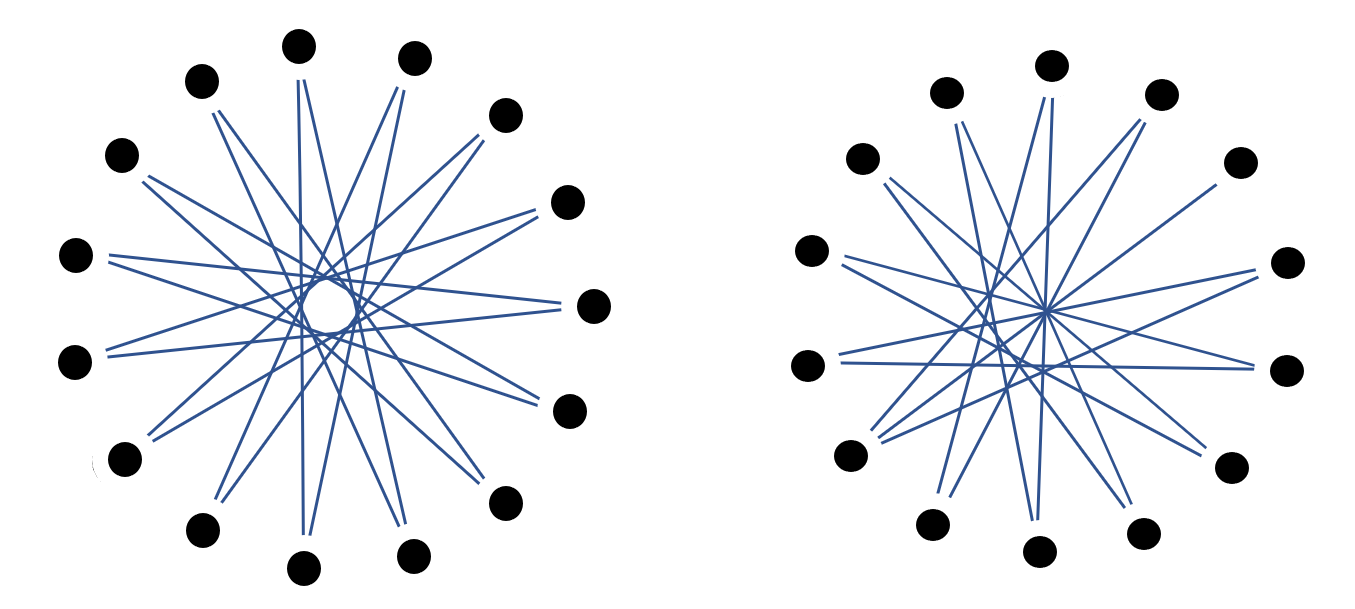}
\caption{Construction of $\mathcal{H}^{++}(n)$}
\label{def2}
\end{figure}

\end{defn}

\begin{defn}[$M_3$-free cghs]\label{m3free} \normalfont
An extremal $M_3$-free construction is simply to take all $n(n - 3)$ triples which contain a pair of cyclically consecutive vertices of $\Omega_n$
plus the set of all  ${n - 4 \choose 2}$ triples without consecutive elements and containing a fixed vertex $v_0$.   It turns out this is not the only $M_3$-free construction with that many edges: we may remove any triple $\{v_0,v_{2k+1},v_{2k+3}\}$ and add $\{v_{2k},v_{2k+2},v_{2k+4}\}$ when $2k + 4 < n$ to obtain many different $M_3$-free extremal constructions.
\end{defn}

\medskip

\begin{defn}[$M_2$-free cghs]\label{m2free} \normalfont
An $M_2$-free construction on $\Omega_n$ is obtained by taking all triples containing a fixed vertex, plus all $n$ triples of three cyclically consecutive vertices.

The restriction $n \geq 7$ is necessary in Theorem \ref{thm:m3m2s3}, since for $n = 6$, the only copies of $M_2$ on $\Omega_6$ are the triples $\{v_0,v_1, v_3\}$, $ \{v_0,v_2, v_3\} $, $\{v_0,v_1, v_4\} $,
$ \{v_0,v_3, v_4\}$, $\{v_0,v_2, v_5\}$, $ \{v_0,v_3, v_5\} $ with their corresponding complements. As such, removing exactly one member from each copy of $M_2$ from the complete cgh on $\Omega_6$ gives an $M_2$-free cgh $H$ with $14 = {n \choose 2} - 1$ triples. It is likely the case that the star plus the set of triples of consecutive vertices in $\Omega_n$ is the unique extremal $M_2$-free example up to isomorphism for $n \geq 8$. For $n=7$, we may take all seven cyclically consecutive triples, the edge $\{v_1,v_3,v_6\}$, and all edges which contain $v_0$ besides the edge $\{v_0,v_4,v_5\}$. Similarly, when $n=7$, we may also take all seven cyclically consecutive triples, the edges $\{v_1,v_3,v_6\}$ and $\{v_1,v_4,v_6\}$, and all edges which contain $v_0$ besides the edges $\{v_0,v_4,v_5\}$ and $\{v_0,v_2,v_3\}$.

\end{defn}

\medskip

\begin{defn}[$S_3$-free cghs]\label{s3free} \normalfont
For even $n \geq 4$, let
\[ \h_0 :=\left\{ \{v_{2i-1},v_{2i},v\} \in \binom{\Omega_n}{3} : 1 \leq i \leq  n/2, v\in \Omega_n\setminus \{ v_{2i-1},v_{2i}\} \right\}. \]
By inspection, $\h_0$ is $S_3$-free and has $n(n - 2)/2$ edges. For $n$ odd, let $\h_1$ have vertex set $\{v_0,v_1,\dots,v_{n-1}\}$ and
add to a copy of $\h_0$ on $\{v_1,v_2,\dots,v_{n - 1}\}$ all triples $\{v_0,v_{2i-1},v_{2i}\}$ where $1 \leq i \leq (n - 1)/2$ as well as
$\{v_{n-1},v_0,v_1\}$. Then $\h_1$ is $S_3$-free and $|\h_1| = (n - 1)(n - 2)/2 + 1$.
\end{defn}

\medskip

\begin{defn}[$S_2$-free cghs]\label{s2free} \normalfont
A construction demonstrating the lower bound is to split $\Omega_n$ into two
intervals $A$ and $B$, and to take all triples which contain a point from $A$ and a pair of consecutive points in $B$.
We also add all triples containing three consecutive points in $B$. This configuration has
$|A|(|B| - 1) + |B| - 2 = (|A| + 1)(|B| - 1) - 1$ triples and does not contain $S_2$. If $|A| = \lceil n/2 \rceil-1$ and $|B| = \lfloor n/2 \rfloor +1 $ then this configuration has
$\lfloor n^2/4 \rfloor - 1$ triples.
\end{defn}

\medskip

\begin{defn}[$D_2$-free cghs]\label{d2free} \normalfont
For a lower bound on $\cex(n, D_2)$,  start with an $S(n,15,2)$ design -- Wilson~\cite{W} proved these exist whenever $n$ is large enough and satisfies the requisite divisibility conditions, i.e., ${n \choose 2}/{15 \choose 2}$ is an integer, and $n \equiv 1 \mod 14$, i.e., $n \equiv 1, 15, 85, 141 \mod 210$.
The construction is as follows: decompose the $E(K_n)$  into  ${n\choose 2}/  {15 \choose 2}$ complete  $K_{15}$'s. Each corresponds to a convex $15$-gon with vertex set $V = \{w_1,w_2,\dots,w_{15}\}$. Decompose each $K_{15}$ into fifteen triangulations of a convex pentagon
  $w_{i}w_{i+1}w_{i+6}w_{i+8}w_{i+11}$ with diagonals $w_{i}w_{i+6}$ and $w_{i}w_{i+8}$ (indices are mod 15).
The lengths of the sides are $1,5,2,3,4$ and the diagonals are $6$ and $7$, so this is indeed a decomposition with $45$ triangles.
   This construction has size exactly
\[ 45 \cdot \frac{{n \choose 2}}{{15 \choose 2}} = \frac{3}{7}{n \choose 2} \]
whenever $n \equiv 1, 15, 85, 141 \mod 210$, and gives a construction of size $\frac{3}{7}{n \choose 2} - O(n)$ for all $n$.
\end{defn}

\section{Proof of Theorem \ref{thm:allplanar}: tangent triangles, $D_1$}\label{sec:proofofthmd1}

A {\em directed triangle} in a tournament is a triangle $\{x,y,z\}$ with $x \to y \to z \to x$.
Let $T(n)$ be the maximum number of directed triangles in an $n$-vertex {\em tournament}. It was shown by
Moon~\cite{Moon} (see also pages 42--44 in Erd\H{o}s and Spencer~\cite{ESp}) that $T(n) = \dotDelta(n)$ for $n \geq 3$. To see this, every tournament with $n$ vertices of outdegrees
$d_1, \ldots, d_n$ has exactly ${n \choose 3} - \sum_{i=1}^n {d_i \choose 2}$ directed triangles. This is maximized (only) when the outdegreees are as equal as possible. If $n$ is odd, then all $d_i=(n-1)/2$ while if $n$ is even, half of the $d_i$ are $(n-2)/2$ and the other half are $n/2$.
These tournaments are called {\em almost regular}.
Tournaments with these outdegrees can easily be constructed and moreover there are plenty of them when $n$ is large. A short calculation gives the required
 $$T(n) = \dotDelta(n).$$

A directed triangle $\{x,y,z\}$ in the plane with $x\to y \to z \to x$ is oriented {\em clockwise} if $z$ is in the half plane to the right when traversing the segment $[xy]$ from $x$ to $y$. If $\{x,y,z\}$ is not oriented clockwise, then it is oriented {\em counterclockwise}.

{\em Proof of Theorem~\ref{thm:d1}}.\quad
Let $P$ be a set of $n$ points in the plane with no three collinear, and let $\TT$ be a $D_1$-free
family of triangles on $P$, and let $H$ be the corresponding $3$-uniform hypergraph with vertex set $P$.
We will prove that $|\TT| \le T(n)$, which gives Theorem~\ref{thm:d1}.

We define an orientation for each pair $\{ x, y \} \in \partial H$ as follows.
Consider any triangle $\{x,y,z\} \in \TT$.
If the orientation of the triangle  $\{x,y,z\}$ is clockwise then orient the edge $\{x,y\}$ as $x\to y$, and
$y\to x$ otherwise. The main observation is that the orientation of $\{x,y\}$ is uniquely determined. If $\{x,y\}$ belongs to an $\{x,y,z\}$ triangle oriented clockwise and to an $\{x,y,z'\}$ triangle oriented counterclockwise, then these two triangles form $D_1$. We conclude that $|H|$ is at most the number of directed triangles in an orientation of a subgraph of $K_n$ which is at most $T(n)$ as required.
\qed

{\em Extremal families}.\quad
The previous proof shows that $|H|=  \dotDelta(n)$ is only possible if the orientation of the edges of $\partial H$ is an almost regular tournament. There is a one to one correspondence between extremal $D_1$-free cghs (or a $D_1$-free triangle system in general) and almost regular tournaments.

\subsection{Extremal $\{ D_1, S_1\}$-free cghs.}\label{subs:41}

Suppose that a cgh $H$ is $D_1$-free and also $S_1$-free with $|H| =\dotDelta (n)$.
Then $\partial H$ is an almost regular tournament and $H$ is obtained as the family of oriented three-cycles
 in $\partial H$.
We claim that more is true, $H\in \mathcal{H}^{\star}(n)$ as described in Construction~\ref{1}.

First we show that all directed triangles in the tournament have the same orientation.
As a first step, we prove that if two triangles in $H$ have some common vertices then
they have the same orientation.
This is obviously true when they have a common edge because $H$ is $D_1$-free.
Consider first the case when two triangles $T_1, T_2\in \TT$ share a vertex $v_1$ and have opposite orientations.
Then they can form an $S_1$ (which we excluded), or an $S_2$, or an $S_3$.

If they form $S_2$, say $v_1< v_2< \dots < v_5< v_1$ and the two triangles are oriented as
$v_1\to v_2 \to v_5\to v_1$ and $v_1\to v_4\to v_3\to v_1$, then we proceed as follows.
Consider the edge $\{v_4,v_5\}$. Observe that $v_5\to v_4$, otherwise the directed triangles $\{v_1,v_4,v_5\}$ and $\{v_1,v_3,v_4\}$ form a $D_1$.
A similar argument shows $v_3\to v_2$.
Consider the edge $\{v_2,v_4\}$. Now $v_2\to v_4$, otherwise the directed triangles $\{v_1,v_2,v_5\}$ and $\{v_2,v_4,v_5\}$ form a $D_1$.
But then we have found a directed triangle $\{v_2,v_4,v_3\}$ which forms an $S_1$ with $\{v_1,v_2,v_5\}$.
So $T_1$ and $T_2$ cannot form an $S_2$.

If $T_1$ and $T_2$ form an $S_3$, say $v_1< \cdots < v_5< v_1$ and the two triangles are oriented as
$v_1\to v_2 \to v_4\to v_1$ and $v_1\to v_5\to v_3\to v_1$ then we proceed the same way.
Consider the edge $\{v_4,v_5\}$. Observe that $v_4\to v_5$, otherwise the directed triangles $\{v_1,v_5,v_4\}$ and $\{v_1,v_2,v_4\}$ form a $D_1$.
A similar argument shows $v_3\to v_2$.
Consider the edge $\{v_3,v_4\}$. Then $v_3\to v_4$, otherwise the directed triangles $\{v_1,v_2,v_4\}$ and $\{v_2,v_4,v_3\}$ form a $D_1$.
But then we have found the directed triangle $\{v_3,v_4,v_5\}$ which forms a $D_1$ with $\{v_1,v_5,v_3\}$.
So $T_1$ and $T_2$ cannot form an $S_3$.

The above argument implies that the vertex sets $X:= \{ x\in e \in H$, $e$ is oriented clockwise$\}$ and
 $Y:= \{ y\in e \in H$, $e$ is oriented counterclockwise$\}$ are disjoint. So every edge $e\in H$ is contained entirely in $X$ or in $Y$. This gives
 \[  |H|\leq  T(|X|) + T(|Y|) < T(n),
 \]
 a contradiction.

From now on, we may suppose that each directed triangle of $\partial H$ is oriented clockwise.
This implies that for $v_i< v_j< v_k < v_i$ we have $v_k \to v_i$ if $v_j\to v_i$. Indeed, each orientation of an edge comes from a directed triangle, so in case of $v_i \to v_k$ and $v_j\to v_i$ we get two triangles $\{v_i,v_k,v_{k'}\}$ and $\{v_j,v_i,v_{j'}\}$ oriented clockwise so
 $v_i< v_{j'}< v_j < v_k < v_{k'}< v_i$, and these two triangles form an $S_1$, a contradiction.
Summarizing, each $v_i$ has out-edges $v_i\to v_{j}$ for $i< j \leq i+ \lfloor (n-1/2)\rfloor$ and
 in-edges $v_j\to v_i$ for $i-\lfloor (n-1/2)\rfloor \leq j < i$, in other words  $H\in \mathcal{H}^{\star}(n)$. \qed

\section{Proof of Theorems~\ref{thm:all} and~\ref{thm:allplanar}: touching triangles, $S_1$}\label{sec:proofofthms1}

\subsection{Proof of Theorem~\ref{thm:all} for $S_1$.}\label{subs:S1cgh}

We will prove that $\cex(n, S_1) \le \cex(n, D_1) + \lfloor n/2\rfloor \lfloor (n-2)/2\rfloor$ which via Theorem \ref{thm:all} for $D_1$
gives the upper bound in Theorem \ref{thm:all} for $S_1$.
For any cgh $\h$, define a graph $G:= G(\h)$ with $G\subset \partial \h$, called the $D_1$-{\em graph of} $\h$ as the set of
$\{u,v\}$ for which there are $x, y \in \Omega_n$ with $u<x<v<y<u$ and triangles $\{u,x,v\}$ and $\{u,v,y\}$ in $H$. In other words, $\{u,v\}$ has triangles on both sides. This definition can be naturally extended to triangle systems $(P, \TT)$.

Let $\h \subset {\Omega_n \choose 3}$ be a cgh containing no copy of $S_1$.
We claim that the $D_1$-graph $G$ is a matching.
Otherwise, if there are $\{u,v\}$ and $\{v,w\}$ in $G$, then there are $x,y \in \Omega_n$ with $u<x<v<y<w$ and triangles $\{u,x,v\}$ and $\{v,y,w\}$ in $H$ which form $S_1$.
We obtain $|G|\le \lfloor n/2\rfloor$. For each $\{u,v\} \in G$, delete all triangles containing $\{u,v\}$ on the side that has fewer triangles (if both sides have the same number of triangles then pick a side arbitrarily). Altogether we delete at most $|G|\lfloor (n-2)/2\rfloor \le
\lfloor n/2\rfloor \lfloor (n-2)/2\rfloor$ triangles.
Let $H'\subset H$ be the set of triangles that remain.
Since $H'$ is $D_1$-free, $|H'|\leq \cex(n, D_1)$, and we are done.

If $H$ is an extremal $S_1$-free cgh, then $|G|= \lfloor n/2\rfloor$ and each edge $\{u,v\} \in G$ is contained in at least $ \lfloor (n-2)/2\rfloor$ triangles $\{u,v,w\}$ with $u< w< v$ and another at least $ \lfloor (n-2)/2\rfloor$ triangles $\{u,v,z\}$ with $u<  v< z$.
This is only possible if the segments representing the edges of $G$ are pairwise crossing each other inside $\Omega_n$.
In case of even $n$ we have that $G$ consists of the $n/2$ diameters $\{v_i,v_{i+n/2}\}$ of $\Omega_n$, in case of odd $n$ we may suppose that $G= \{ \{v_{j},v_{j + (n - 1)/2} \} : 0 \leq j \leq (n - 3)/2 \}$.
Since $H'$ is an extremal $\{ D_1, S_1\}$-free cgh the results of subsection~\ref{subs:41} yield that
$H'\in \mathcal{H}^{\star}(n)$.
The triples from $H\setminus H'$ can be added to $H'$
only as described in Construction~\ref{3}, and this yields $H\in \mathcal{H}^{+}(n)$. \qed


\subsection{A geometric lemma about $D_1$-edges in $S_1$-free triangle systems.}\label{subs:S1_52}

Write $\tria{uvw}$ for the triangle with vertices $u,v,w$. Recall that a segment $[ab]$ (with $a,b\in P$, $a\neq b$) is a $D_1$-edge in the triangle system $(P, \TT)$  if  there are triangles from $\TT$ on both sides, i.e.,
 $\exists c^-,c^+\in P$ such that $c^-$ and $c^+$ are separated by the line $\lina{ab}$ and $\triangle(abc^-), \tria{abc^+} \in   \TT$.

\begin{lemma}\label{lem:521}
Let $\TT$ be a triangle system with point set $P$.
Suppose that  $[ab]$, $[bc]$, and $[cd]$ are distinct $D_1$-segments in $\TT$ (so $a=d$ is not excluded).
Then $\TT$ contains an $S_1$ configuration.
\end{lemma}

\begin{proof}
The lines $\lina {ab}$, $\lina {bc}$, and $\lina {cd}$ cut the plane into seven open regions, unless $\lina {ab} || \lina {cd}$ when we get only six regions. Let $T$ be the triangle these lines enclose (in the case of six regions $T$ is one of the infinite threesided strips).
Let $H(xy)$ denote the open half plane with boundary line $xy$ tangent to $T$ but disjoint from its interior.
Since $ab$ is a $D_1$-edge there exists a triangle $abc^-\in \TT$ where $c^-$ is in the open half plane   $H(ab)$, and there exists a triangle $bcy\in \TT$  
 with $y\in H(bc)$.
These two triangles form an $S_1$ configuration unless $c^-$ and $y\in B:= (H(ab)\cup \lina{ab}) \cap H(bc)$.
Consider a third triangle  $cdb^-\in \TT$ where $b^-\in H(cd)$.
Since this half plane is separated from $bcy$ by $\lina{cd}$ (except both contain $c$ in their boundaries)  $\tria{cdb^-}$  and $\tria{bcy}$ form an $S_1$, and we are done.
\end{proof}

\subsection{A removal lemma concerning  $S_1$-free triangle systems.}\label{subs:S1_53}
We prove Theorem~\ref{thm:allplanar} for $S_1$ in the following stronger form.

\begin{thm}\label{thm:s1_to_d1}
Let $n \geq 3$, and let $\TT$ be an $n$-point triangle system.
If $\TT$ is $S_1$-free then there exists a subfamily $\TT'\subset \TT$ which is $D_1$-free and
\[  |\TT|\leq |\TT'|+
\Big\lfloor\frac{n}{2}\Big\rfloor
\Big\lfloor\frac{n-2}{2}\Big\rfloor. \]
\end{thm}
Since $|\TT'|\leq \dotDelta(n)$ by Theorem~\ref{thm:d1}, one obtains the desired upper bound for $|\TT|$.

Recall that the $D_1$-graph of $\TT$ is a graph $G$ with vertex set $P$ and its edges are the $D_1$-segments.
For $e \in G$, let $\TT(e)$ be the set of triangles from $\TT$ containing $e$, and let $\de{e}$ be the minimum number of triangles $\tria{e \cup \{x\}}\in \TT$ on one side of $\lina{e}$.
Obviously, $\de{e}\leq (1/2)|\TT(e)|\leq \lfloor (n-2)/2 \rfloor$.
We extend this definition for any set of pairs, $\TT(F)$ is the set of triangles from $\TT$ containing a pair
$e\in F$, and $\de{F}$ is the minimum number of triangles $e \cup \{x\}\in \TT$, $e\in F$ such that
 removing those triangles from $\TT$ we eliminate all $D_1$ edges of $F$.
Our aim is to prove that $\de{G}\leq \lfloor n/2\rfloor \lfloor (n-2)/2 \rfloor$.
We also show that for $n\neq 5$ in case of equality $G$ is either a matching of size $\lfloor n/2\rfloor$, or
 a matching of size $(n-3)/2$ and a path of length two. We conjecture that the latter case cannot happen for $n>n_0$.

Since $\TT$ is $S_1$-free, Lemma~\ref{lem:521} implies that $G$ contains no path of length three and in particular $G$ does not contain a cycle. Thus $G$ is a starforest.

\begin{claim}\label{cl:531}
Suppose that $\{e,f\}\subset G$ is a two-edge component of the $D_1$-graph $G$,
Then $\de{e,f}\leq  \lfloor (n-2)/2 \rfloor$.
\end{claim}
\begin{proof}
We have that there exists a $w\in P$, $w:=e \cap f$.
Let $\delta = 1$ if $\tria{e \cup f} \in \TT$, and $\delta = 0$ otherwise.
We assume $e,f$ are as shown in Figure~\ref{fig:s1_new}, i.e.,
 the lines $\lina e$ and $\lina f$ cut the plane into four open regions $A,B,C,D$ ($B$ is disjoint to $e\cup f $, the boundary of $D$
 contains both, etc.).
If any of the dotted triples is in $\TT$, then we get a copy of $S_1$, as we have seen this in the proof of Lemma~\ref{lem:521}.
More formally, we get, e.g., if $e\cup \{ x\}\in \TT$, then either $e\cup \{ x\}=e\cup f$ or
 $x\in B\cup C$.
For $X \in \{A,B,C,D\}$, let $e_X$ be the number of $x \in X$ such that $e \cup \{x\} \in \TT$.
We have $e_A, e_D=0$ and $f_C, f_D=0$.
Observe that $e \cup \{x\},f \cup \{x\} \in \TT$ is not possible for $x \in B$, else we get a $D_1$-edge $[wx]$, contradicting the fact that $\{ e,f\}$ is a component of $G$.
We obtain
\begin{equation}\label{eq:sectors}
f_A + e_B + f_B + e_C \leq |P| - |\{ e\cup f\}|\leq n - 3.
\end{equation}

\begin{figure}[H]
\centering
\includegraphics[scale=0.8]{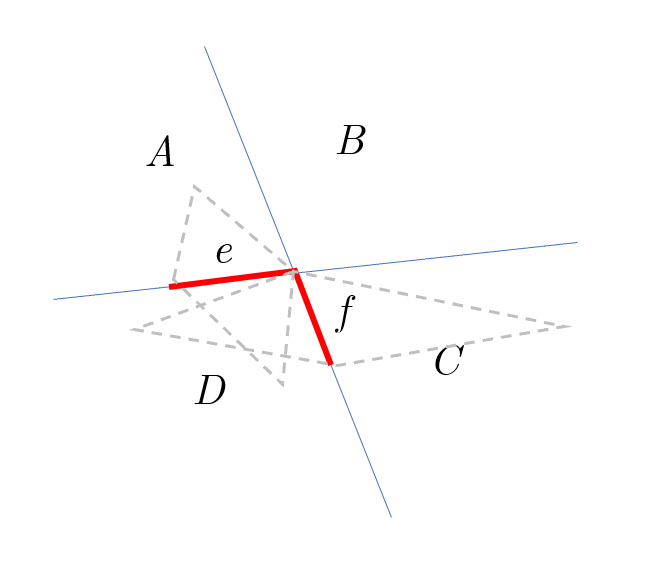}
\caption{A two edge component of $G$ as discussed in Claim~\ref{cl:531} }
\label{fig:s1_new}
\end{figure}

There are four possibilities to delete edges from $\TT$ to make $\{ e,f\}$ non-$D_1$-edges, namely we can eliminate all triangles $e\cup \{ x\}$ with  $x\in A\cup B$ or all such triangles from the other side of $\lina e$, and there are two sides of $\lina f$ as well.
We get four inequalities for $\de{e,f}$.
\begin{eqnarray*}
  \de{e,f} &\leq& e_B+ f_B\\
 \de{e,f}  &\leq& e_B+ (f_A+\delta)\\
  \de{e,f} &\leq & (e_C+\delta)+ f_B\\
  \de{e,f}  &\leq & (e_C+\delta)+ f_A.
\end{eqnarray*}
Summing these and using~\eqref{eq:sectors} we get
$4 \de{e,f} \leq 2n-6 + 3\delta \le 2n-3$. This gives
 $\de{e,f}  \leq \lfloor(2n-3)/4\rfloor = \lfloor(n-2)/2\rfloor$ and we are done.
\end{proof}

\begin{claim}\label{cl:532}
Suppose that $F\subset G$ is a component of the $D_1$-graph $G$, a star with $s\geq 3$ edges.
Then $\de{F}\leq \lfloor (n-1)/2 \rfloor$.
\end{claim}
\begin{proof}
We will prove the stronger statement $|\TT(F)|\leq n-1$.
Suppose that the edges of $F$ are $wv_1$, $wv_2, \dots, wv_s$.
We claim that for any vertex $x\in P\setminus \{ w\}$ an (open) half plane with boundary line $\lina{wx}$ can contain only at most one
 triangle from $\TT$ of the form $wxv_i$. Indeed, if there is another such triangle $wxv_j$ and, say, $\angle (xwv_i) < \angle(xwv_j)$ then there is another vertex $z\in P$ such that $\tria{wv_jz}\in \TT$ and it is separated from $ \tria{wxv_i}$ by the line $\lina{wv_j}$; however this means that $\tria{wv_jz}$ and $ \tria{wxv_i}$ form an $S_1$ configuration.
Even more, if $x\in P\setminus V(F)$, then $[wx]\notin G$ implies that this can happen on at most one side of $\lina{wx}$. We get for such an $x$ that
 $|\{ wv_i: wv_ix\in \TT\} |\leq 1$, hence $|\{ wv_ix: wv_ix\in \TT, 1\leq i\leq s, x\notin V(F)\}|\leq n-1-s$.
To estimate $|\TT(F)|$ it remains to count  the triangles from $\TT$ of the form $wv_iv_j$.
For any given $i$ there are at most two such triangles, and each of them is counted that way exactly twice, so their number is at most $s$.
\end{proof}

\begin{proof}[Proof of Theorem~\ref{thm:s1_to_d1}]
Suppose that $\de{G}\geq \lfloor n/2\rfloor \lfloor (n-2)/2 \rfloor$.
Let the (nontrivial) components of $G$ be $F_1, F_2, \dots, F_r$.
Claims~\ref{cl:531} and~\ref{cl:532} imply that
$\de{G} =\sum \de{F_i} \leq r\lfloor (n-1)/2 \rfloor$.
For $n$ even this leads to $r\geq n/2$; equality holds, $G$ is a perfect matching.
For $n\geq 5$ odd we get $r\geq (n-3)/2$ and in case of $r=(n-3)/2$ we have
  $\de{F_i}=(n-1)/2$ for each $1\leq i\leq r$.
In this latter case again  Claim~\ref{cl:531} implies that each $F_i$ has at least 4 vertices, $r\leq n/4$, a contradiction for $n>5$. So in the odd case (for $n>5$) we must have $r=(n-1)/2$, each component is a single edge except perhaps one is a two-path.
Then Claim~\ref{cl:531} implies that $\de{G}\leq r \lfloor (n-2)/2 \rfloor$, completing the proof.
\end{proof}

\section{Proof of Theorems~\ref{thm:all} and~\ref{thm:allplanar}: two separated triangles, $M_1$}\label{sec:proofofthmm1}

\subsection{Proof of Theorem~\ref{thm:all} for $M_1$.}
	
	We use a method similar to that in~\cite{FHK} to
determine $\cex(n,M_1)$.
We prove that if $H$ is an $n$-vertex
$M_1$-free cgh with $|H|\geq \dotDelta(n) + n(n-3)/2$, then $H \in \mathcal{H}^{++}(n)$. First let $n \geq 3$ be odd. If $H \in \mathcal{H}^{++}(n)$ then we are done,
so we may assume $H$ contains a triangle $T(i,j,k) = \{v_i, v_j, v_k\}$ with
$v_i < v_j < v_k < v_{i + (n - 1)/2}$. Moreover, we may assume that among all such triangles, $T(i,j,k)$ is the triangle where the longest edge $\{v_i,v_k\}$ is as short as possible.  Replace all triangles $T(i,j',k) \in H$ with $i<j'<k$  with all triangles $T(i-1,k+1, l)$ where $j$ and $l$ are on opposite sides of the edge $\{v_i,v_k\}$ as shown in Figure~\ref{m1proof.png}. Since $T(i,j,k)$ and $T(i-1,k+1, l)$ form a copy
of $M_1$, $T(i-1,k+1, l) \not \in H$ for all such $l$. Moreover, since $v_i < v_k < v_{i + (n - 1)/2}$, the number of triangles $T(i-1,k+1, l)$ that we added is greater than the
number of triangles $T(i,j,k)$ that we deleted. Consequently,
this produces a cgh $H'$ with $|H'| > |H|$.  Since $H$ is extremal $M_1$-free, there exists a copy of $M_1$ in $H'$, which must contain a triangle $T(i-1,k+1, l) \in H'$.
 Since all triangles $T(i-1,k+1,l)$ intersect, the other triangle in the copy of $M_1$ must be $T(f,g,h) \in H$. Since $H$ is $M_1$-free, $T(f,g,h)$ intersects $T(i,j,k)$, which implies $v_i \leq v_f < v_g < v_h \leq v_k$ and $\{v_f,v_h\} \neq \{v_i,v_k\}$. However, then the edge $\{v_f ,v_h\}$ is shorter than the edge $\{v_i, v_k\}$, a contradiction.

\begin{figure}[ht]
\centering
\includegraphics[scale=.3] {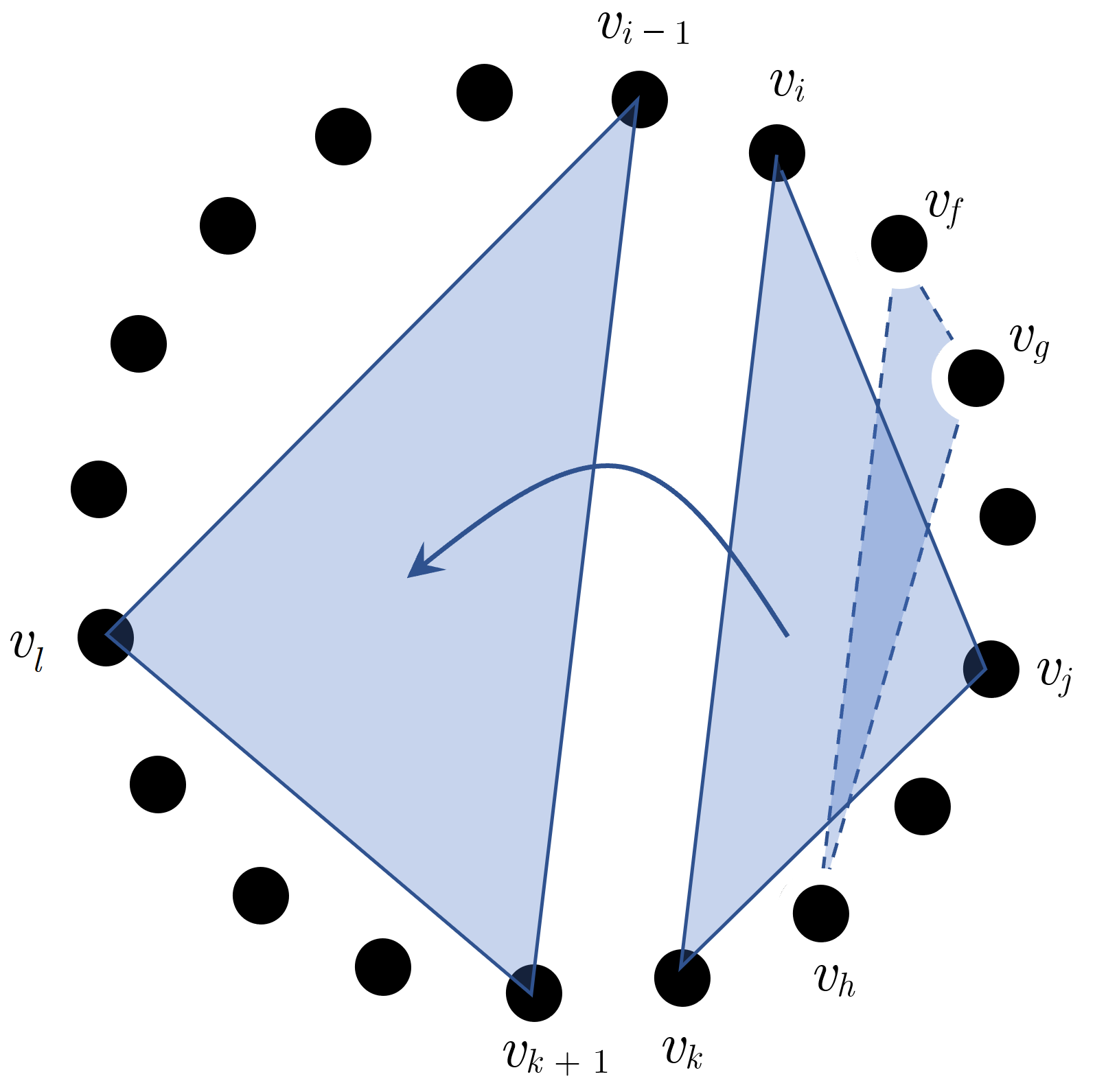}
\caption{Replacing triangles in an $M_1$-free cgh}
\label{m1proof.png}
\end{figure}

Now let $n \geq 4$ be even and let $H$ be an extremal $n$-vertex $M_1$-free cgh. If $H \in \mathcal{H}^{++}(n)$ we are done, so suppose $H \not \in \mathcal{H}^{++}(n)$. If $H$ contains a triangle $T(i,j,k)$ where $v_i < v_j < v_k < v_{i + n/2 - 1}$, then we repeat the same proof as in the case $n$ is odd to derive a contradiction.
Therefore all triangles in  $H$ contain the centroid or are $T(i,j,k)$ with $v_i < v_j < v_k = v_{i + n/2 - 1}$. The pairs $\{v_i,v_{i + n/2 - 1}\}$ for which there exists such a triangle $T(i,j,k)$ must pairwise intersect (possibly at their endpoints) otherwise we find a copy of $M_1$ in $H$. In particular, by definition of Construction~\ref{2}, $H \in \mathcal{H}^{++}(n)$. \qed

Let us note that we can give another proof using the $D_1$-graph and Theorem~\ref{thm:all} for $D_1$ just as we did in subsection~\ref{subs:S1cgh} to prove Theorem~\ref{thm:all} for $S_1$-free convex triangle systems -- specifically, the graph of $D_1$-pairs does not contain two geometrically disjoint pairs.

\subsection{Proof of Theorem~\ref{thm:allplanar} for $M_1$.}\label{sec:proofofthmm1_gen}

We prove Theorem~\ref{thm:allplanar} for $M_1$ in the following stronger form.

\begin{thm}\label{thm:m1_to_d1}
	Let $n \geq 3$, and let $\TT$ be an $n$-point triangle system.
	If $\TT$ is $M_1$-free then there exists a subfamily $\TT'\subset \TT$ which is $D_1$-free and
	\[  |\TT|\leq |\TT'|+  C_V {n \choose 2}. \]
\end{thm}
Since $|\TT'|\leq \dotDelta(n)$ by Theorem~\ref{thm:allplanar} for $D_1$, one obtains the desired upper bound
$|\TT|\leq \dotDelta(n) +O(n^2)$.
Here $C_V>0$ is a constant obtained from Theorem~\ref{th:Valtr} below due to Valtr.

A {\em geometric graph} $(V,E)$ is a graph drawn in the plane so that the vertex set $V$ consists of points in general position and the edge set $E$ consists of straight-line segments between points of $V$.
Two edges of a geometric graph are said to be {\em avoiding}, if they are opposite sides of a
convex quadrilateral.
\begin{thm}[Valtr~\cite{Valtr}]\label{th:Valtr}
	There is a constant $C_V>0$ such that any geometric graph on $m$ vertices with no three pairwise avoiding edges has at most  $C_Vm$ edges.
\end{thm}

Recall that a segment $[ab]$ (with $a,b\in P$, $a\neq b$) is a $D_1$-edge in the triangle system $(P, \TT)$  if  there are triangles from $\TT$ on both sides, i.e.,
$\exists c^-,c^+\in P$ such that $c^-$ and $c^+$ are separated by the line $\lina{ab}$ and the triangles $\tria{abc^-}$ and $\tria{abc^+}\in \TT$.
The set of all such segments is the $D_1$-graph $G$ of $\TT$.
For $v\in P$ let $G_v$ be the $D_1$-{\em link graph} of $\TT$, i.e, it consists of those edges $e$ of $G$, $v\notin e$,  which are contained in a triangle $\tria{e\cup \{ v\}}\in \TT$.
The vertex set of the geometric graph $G_v$ is $P\setminus \{ v\}$, and for every edge $e\in G_v$ we can choose a triangle $\tria{e, -v}\in \TT$ which is separated from the triangle $\tria{e,v}$ by the line $\lina{e}$, so the third vertex
of $\tria{e, -v}$ and $v$ lie on different sides of $\lina{e}$.

\begin{lemma}\label{lem:71}
	Let $\TT$ be a triangle system with point set $P$, and let  the three segments $e$, $f$, and $g$ of $E(G_v)$ be pairwise avoiding.
	Then $\TT$ contains $M_1$.
\end{lemma}

\begin{proof}
	Given a line $\ell$ and a set $X\neq \emptyset$ with $X\cap \ell=\emptyset$ we denote the open half plane with boundary $\ell$ and  containing $X$ by $H(\ell, X)$, the other side is $H(\ell, -X)$.
	Suppose that $\TT$ contains no disjoint triangles.
	Since $e$ and $f$ are on opposite sides of a convex quadrilateral,
	the triangle $\tria{e, -f}\in \TT$ should meet $\tria{f, v}$. This is only possible if
	$v\in H(\lina{e}, -f)$.
	Similarly,  $v\in H(\lina{f}, -e)$, so $v$ is in the  open wedge $H(\lina{e}, -f) \cap H(\lina{f}, -e)$, cf., Figure~\ref{fig:s1_new}.
	For later use denote this wedge by $B(e,f)$.
	Since $v\in B(e,f)$, this rules out that the lines $\ell(e), \ell(f)$ are parallel.

	The line $\lina{ f}$ avoids the other two segments, suppose that it separates them, i.e.,
	$e\subset H(\lina{ f}, -g)$ (and  $g\subset H(\lina{ f}, -e)$).
	Then $B(e,f)\subset H(\lina{ f}, -e)$ and $B(f,g)\subset H(\lina{ f}, -g)= H(\lina{ f}, e)$.
	This implies $B(e,f)\cap B(f,g)=\emptyset$, contradicting to $v\in B(e,f)\cap B(f,g)\cap B(g,e)$.
	Hence $\lina{e}$ is a tangent line of $R:={ \, {\mathrm{conv}} \, ( \{ e,f,g\}) }$, so this convex hull is a hexagon.
	
	There are two cases. If $R$ is inscribed into the triangle $T$ formed by the lines $\lina{ e}$, $\lina{ f}$, and $\lina{ g}$, then
	each region $B(e,f)$, $B(f,g)$, and $B(g,e)$ is a digon (an infinite wedge). These are pairwise disjoint, there is no place for $v$.
	Otherwise, one edge, say $e$ lies on a side of $T$ and $f$ and $g$ lie on the other two sides of the threesided infinite region
	$H(\lina{e},-T)\cap H(\lina{f},g)\cap H(\lina{g},f)$. Then $B(f,g)$ is a digon inside $H(\lina{e}, T)$, and $v\in B$.
	Consider a triangle $\tria{e,x}\in \TT$ where $x\in H(\lina{e},-T)$.
	The two digons in $H(\lina{e},-T)$ are disjoint, so we may suppose that $x\notin (H(\lina{e},-T)\cap H(\lina{f},-T))$.
	Then the triangle  $\tria{e,x}$ is disjoint to $\tria{f,v}$, completing the proof of Lemma~\ref{lem:71}.
\end{proof}

\begin{proof}[Proof of Theorem~\ref{thm:m1_to_d1}]
	Recall that we denote  the set of triangles from $\TT$ containing a pair
	$e\in F$  by $\TT(F)$, and $\de{F}$ is the minimum number of triangles $e \cup \{x\}\in \TT$, $e\in F$ such that
	removing those triangles from $\TT$ we eliminate all $D_1$ edges of $F$.
	Our aim is to prove that $\de{G}\leq  C_V {n \choose 2}$ if $\TT$ is $M_1$-free.
	We will show the slightly stronger statement: $\TT(G)\leq C_V n(n-1)$.
	We have $\TT(G) \leq \sum_{v\in V} |G_v|$.
	By Lemma~\ref{lem:71} the geometric graph $G_v$ has no three pairwise avoiding edges. Then Theorem~\ref{th:Valtr}
	gives $|G_v|\leq C_V(n-1)$. Then $\de{G}\leq (1/2)|\TT(G)|$ completes the proof.
\end{proof}

\section{Proof of Theorem~\ref{thm:m3m2s3}: crossing triangles, $M_3$}\label{sec:proofofthmm2m3}

For the proof of Theorem~\ref{thm:m3m2s3} for $M_3$, it is useful to consider ordered hypergraphs: the vertex set is $\Omega_n = \{v_0,v_1,\dots,v_{n-1}\}$ with the linear ordering $v_0 < v_1 < \dots < v_{n-1}$. Let $\ex_{\to}(n,M_3)$ denote the maximum number of triples in an ordered hypergraph not containing triples
$\{v_i,v_j,v_k\}$ and $\{v_{i'},v_{j'},v_{k'}\}$ with $v_i < v_{i'} < v_j < v_{j'} < v_k < v_{k'}$ -- this is the ordered analog of $M_3$.
The following theorem implies Theorem \ref{thm:m3m2s3} for $M_3$, since $\cex(n,M_3) = \ex_{\to}(n,M_3)$:

\begin{thm}\label{cpthm}
Let $n \geq 7$. Then  $\ex_{\to}(n, M_3)=  {n \choose 3} - {n - 3 \choose 3}$.
\end{thm}

\begin{proof}
Let $H$ be an $M_3$-free ordered triple system with $n$ vertices. Let $H_1$ consists of all $e \in H$ with $v_0,v_1 \in e$, and let
$H_2$ consists of all $e \in H$ with $v_0\in e$, $v_1\not\in e$ and $e \setminus  \{ v_0\} \cup \{ v_1\} \in H$.
Let $H_3$ be obtained from $H \backslash (H_1 \cup H_2)$ by merging the vertices $v_0$ and $v_1$. Note
that $H_3$ is a 3-cgh with $n - 1$ vertices. Clearly, $|H_1| \leq n - 2$. We may form an ordered
graph from $H_2$ by considering $G = \{\{u,v\} : \{v_0,u,v\} \in H_2\}$ -- this is the {\em link graph of $v_0$}
with vertex set $\{v_2,v_3,\dots,v_{n-1}\}$ with the natural ordering. If two edges of $G$ cross --
say $\{u,v\},\{w,x\} \in G$ with $u < w < v < x$, then the triples $\{u,v,v_1\}$ and $\{w,x,v_0\}$ are in $H_2$,
and form a copy of $M_3$, a contradiction. Therefore no two edges of $G$ cross, which implies
$G$ is an outerplane graph with $n - 2$ vertices. Consequently $|G| \leq 2n - 7$, by Euler's Formula.
Finally, it is also straightforward to check $H_3$ is $M_3$-free, so by induction,
\[ |H| = |H_1| + |H_2| + |H_3| \leq (n - 2) + (2n - 7) + {n - 1 \choose 3} - {n - 4 \choose 3} = {n \choose 3} - {n - 3\choose 3}.\]
This completes the proof of Theorem \ref{cpthm}.
\end{proof}

\section{Proof of Theorem \ref{thm:m3m2s3}: stabbing triangles, $M_2$}\label{sec:proofofthmm2m2}

We prove by induction on $n$ that $\cex(n,M_2) = {n \choose 2} - 2$ for $n \geq 7$. When $n=7$, since cyclically consecutive triples $\{v_i,v_{i+1},v_{i+2}\}$ are never in $M_2$, we may assume these seven edges are in any $M_2$-free cgh. For the remaining twenty-eight triples, we create a graph with vertex sets consisting of these triples and form an edge if two of the triples form a copy of $M_2$. A computer aided calculation \cite{SAGE} then yields this graph has independence number $12$ and hence $\cex(7,M_2) = 12 + 7 = \binom{7}{2}-2$.

\medskip

For the induction step, we plan to find two consecutive $u,v \in \Omega_n$ with degree at most three and whose common link graph $G_{u} \cap G_{v}$ has at most $n - 3$ edges. Let $H$ be a maximal $M_2$-free cgh on $\Omega_n$, and $H' \subset H$ be the cgh after removing all consecutive triples $\{v_i,v_{i+1},v_{i+2}\}$. Let $d(v_i,v_j)$ be length of the path on the perimeter of the polygon starting with $v_i$ and moving clockwise to $v_j$. For an edge $e = \{v_i, v_{i+1}, v_k\} \in H'$ -- we only consider such edges -- let $\ell(e) = \min\{d(v_{i + 1},v_k),d(v_k,v_i)\}$.

\begin{lemma}\label{lemma:m2induction}
Let $H \subset \binom{\Omega_n}{3}$ be a maximal $M_2$-free cgh and $H'$ be as above. Then \\
\textup{(1)}  For consecutive $u,v \in \Omega_n$, $|G_{u} \cap G_{v}| \leq n - 3$ with equality only if $G_u \cap G_v$ is a star. \\
\textup{(2)} There exists $v_i \in \Omega_n $ such that the degree of $\{v_{i},v_{i + 1}\}$ is at most three in $H$.
\end{lemma}

\begin{proof}
We first prove (1) by showing $G_{u,v} := G_u \cap G_v$ does not contain a pair of disjoint edges If $\{w,x\},\{y,z\}$ are disjoint edges in $G_{u,v}$, and $v < w < x < y < z < u < v$ or $v < w < y < z < x < u < v$ -- this means that $\{w,x\},\{y,z\}$ do not cross -- then $\{u,w,x\},\{v,y,z\}$ form $M_2$. If on the other hand $v < w < y < x < z < u < v$ -- this means $\{w,x\}, \{y,z\}$ do cross --
then $\{u,y,z\},\{v,w,x\}$ form $M_2$. So $G_{u,v}$ has no pair of consecutive edges. It is a standard fact that the unique extremal graphs with at least four vertices and no pair of disjoint edges are stars, and therefore $G_{u,v}$ has at most $n - 3$ edges.

For (2), seeking a contradiction, suppose every pair of consecutive vertices has degree at least four in $H$ and hence degree at least two in $H'$.
We first show there exists $e \in H'$ with $\ell(e) \geq 3$. If not, then $\{v_i, v_{i + 1}, v_{i + 3}\} \in H'$ and $\{v_{i - 2}, v_i, v_{i + 1}\} \in H'$ for all $i$ and there are no other edges in $H'$. However, then $\{v_0, v_1, v_3\}  \in H'$ and $\{v_2, v_4, v_5\}  \in H'$ form $M_2$, a contradiction. So there exists $e \in H'$ with $\ell(e) \geq 3$.
From all $e \in H'$ with $\ell(e) \geq 3$, pick $e$ so that $\ell(e) = j \geq 3$ is a minimum. Suppose $e = \{v_0,v_1, v_{j + 1}\}$, so $\ell(e) = d(v_1,v_{j + 1})$ (the proof
for $e$ of the form $\{v_{n - j}, v_0, v_1\}$ with $\ell(e) = j = d(v_{n-j},v_0) \geq 3$ will be symmetric). Then the pair $\{v_{j - 1},v_{j}\}$ has degree at least two in $H'$ so there are edges $f = \{v_h,v_{j-1},v_j\}$ and $g = \{v_k,v_{j-1},v_j\}$ in $H'$. If $j + 1 < k \leq n - 1$ or $j + 1 < h \leq n - 1$, then
$f$ and $e$ or $g$ and $e$ respectively form $M_2$, a contradiction. So $0 \leq h,k \leq j - 3$, recalling $\{v_{j-2},v_{j-1},v_j\} \not \in H'$.
Now
\[ \ell(f) = d(v_h,v_{j-1}) > d(v_k,v_{j - 1}) \geq 2\]
and so $\ell(f) \geq 3$. On the other hand, since $0 \leq h < j - 1$,
\[ \ell(f) = d(v_h,v_{j - 1}) < d(v_0,v_j) = \ell(e)\]
contradicting the choice of $e$. This final contradiction proves (2).
\end{proof}

Let $\{v_i,v_{i + 1}\}$ have degree at most three in $\h$, as guaranteed by Lemma \ref{lemma:m2induction} part (2). We contract the pair $\{v_i,v_{i+1}\}$ to a vertex $w$ to get a cgh $H_0$ with $n - 1$ vertices. Let $G = \{\{u,v\} : \{u,v,v_i\},\{u,v,v_{i+1}\} \in H \}$ be the common link graph of $v_i$ and $v_{i + 1}$.

\begin{lemma}\label{lemma:G}
Let $G$ be the common link graph of $v_i$ and $v_{i + 1}$. Then $|G| \leq n - 4$.
\end{lemma}

\begin{proof}
If neither of $\{v_{i-1},v_i,v_{i+2}\}$ or $\{v_{i-1},v_{i+1},v_{i+2}\}$ is in $H$, then $\{v_{i-1},w,v_{i+2}\} \not \in H_0$ and $|G| \leq n - 4$ follows from Lemma \ref{lemma:m2induction} part (1). So we assume $\{v_{i-1},v_i,v_{i+2}\} \in H$ or $\{v_{i-1},v_{i+1},v_{i+2}\} \in H$.

{\bf Case 1.} $\{v_{i-1},v_i,v_{i+2}\} \in H$. Suppose $G$ is a star with $n - 3$ edges, with center $v_k$.  If $v_k \notin \{v_{i-1}, v_{i+2} \}$, then letting $v_j \notin \{ v_k,v_{i-1},v_i, v_{i+1}, v_{i+2} \}$, it follows that $\{v_i,v_j,v_k\}$ and $\{v_{i-1},v_{i+1},v_{i+2}\}$ form a copy of $M_2$. Hence, we may assume that $v_k= v_{i-1}$ or $v_k= v_{i+2}$. Both of these cases are similar, so consider only the case $v_k=v_{i+2}$. We may assume that $\{ v_{i+3}, v_{i+4} \}$ has degree at least three. Then there is at least one triple which contains $\{ v_{i+3}, v_{i+4} \}$ of the form $\{v,v_{i+3},v_{i+4}\}$. If $v \in \Omega_n$ and $v_{i+4} < v < v_{i+1}$, then $\{v,v_{i+3},v_{i+4}\}$ and $\{v_{i+1},v_{i+2},v_{i+5}\}$ form $M_2$. If $v=v_{i+1}$, then $\{v,v_{i+3},v_{i+4}\}$ and $\{v_{i-1},v_i,v_{i+2}\}$ form $M_2$. So $G$ is not a star
with $n - 3$ edges, and Lemma \ref{lemma:m2induction} part (1) gives $|G| \leq n - 4$.

{\bf Case 2.} $\{v_{i-1},v_{i+1},v_{i+2}\}\in H$. In this case, a symmetric argument to that used for $\{v_{i-1},v_i,v_{i+2}\} \in H$ applies by reversing the orientation of $\Omega_n$.
\end{proof}

To complete the proof of $|H| \leq {n \choose 2} - 2$, we note by inspection that $H_0$ is also $M_2$-free. By induction, $|H_0| \leq {n - 1 \choose 2} - 2$.
By Lemma \ref{lemma:G}, and recalling $d_H(v_i,v_{i + 1}) \leq 3$,
\[ |H| = |H_0| + |G| + d_H(v_i,v_{i + 1}) \leq {n - 1 \choose 2} - 2 + n - 4 + 3 = {n \choose 2} - 2.\]
This proves Theorem \ref{thm:m3m2s3} for $M_2$. \qed

\section{Proof of Theorem \ref{thm:m3m2s3}: crossing triangles sharing a vertex, $S_3$}\label{sec:configs3}

Let $\h \subset \binom{\Omega_n}{3}$ be a $S_3$-free cgh and $G_i$ be the link graph of $v_i$ in $\h$. Let $G_i'$ comprise the edges of $G_i$ which consist of two consecutive vertices in $\Omega_n$, and let $G_i'' = G_i \backslash G_i'$.

\begin{lemma}\label{lemma:noncrossdist2}
Let $\h \subset \binom{\Omega_n}{3}$ be a $S_3$-free cgh. For $0 \leq i \leq n - 1$, $|G_i''| \leq n - 3$.
\end{lemma}

\begin{proof}
The graph $G_i''$ has no pair of crossing edges since $\h$ is $S_3$-free. If we add to $G_i''$ all the $n$ edges $\{v_j,v_{j+1}\}$, we obtain a subdivision (maybe a triangulation) of $\Omega_n$. A triangulation has
$2n-3$ edges. Removing the $n$ added edges gives $|G_i''| \leq n - 3$.
\end{proof}

\begin{lemma}\label{lemma:dist1}
Let $\h \subset \binom{\Omega_n}{3}$ be a $S_3$-free cgh. For each $i$  $|G_{i}'| + |G_{i+1}'| \leq n$.
\end{lemma}

\begin{proof}
We may assume $i = 0$. Let $G$ denote the multigraph obtained by superimposing the graphs $G_0'$ and $G_1'$, so $|G| = |G_0'| + |G_1'|$. Each component $C$ of $G$ is a path $P$ with some edges of multiplicity two. If $\{v_{j-1},v_j\} \in P \cap G_0'$, then $\{v_j,v_{j+1}\} \not \in P \cap G_1'$, otherwise $\{v_0,v_j,v_{j+1}\},\{v_1,v_{j-1},v_j\}$ form $S_3 \subset \h$ as in Figure~\ref{s3config}, a contradiction. If all edges of $P$ are from $G_1'$ only, then $|C| = |P| = |V(C)| - 1$.
Otherwise, let $\{v_{j},v_{j+1} \}$ be the first edge  of $P$ in $G_0'$ in the clockwise direction. Then all edges of $P$ preceding $\{v_j,v_{j+1}\}$ are in $G_1'$ only,
and all edges of $P$ after $\{v_j,v_{j+1}\}$ are in $G_0'$ only, whereas $\{v_j,v_{j+1}\}$ might be in both $G_0'$ and in $G_1'$. Therefore at most one edge of $P$ has multiplicity two, and $|C| \leq |P| + 1 = |V(C)|$. If $C_1,C_2,\dots,C_r$ are the components of $G$, we conclude
$|G| = |C_1| + |C_2| + \dots + |C_r| \leq |V(C_1)| + |V(C_2)| + \dots + |V(C_r)| = |V(G)| = n$.
\end{proof}

\begin{figure}[H]
\centering
\includegraphics[scale=0.5]{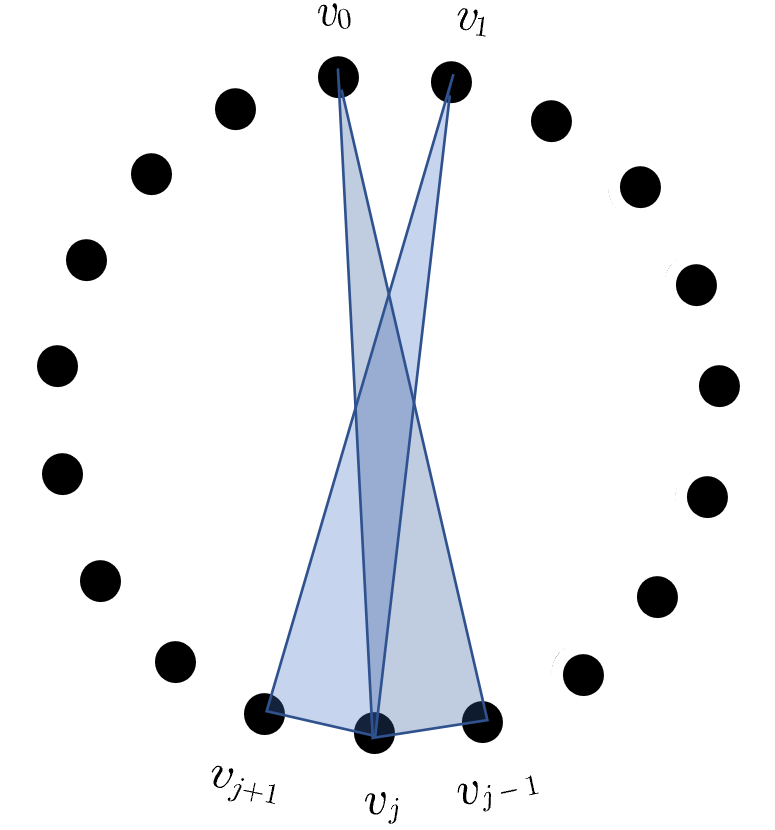}
\caption{Crossing triangles in the proof of Lemma~\ref{lemma:dist1}}
\label{s3config}
\end{figure}

We now complete the proof of $\cex(n,S_3) \leq n(n - 2)/2$, using the following identity:
\[
3|\h| = \sum_i (|G_i'| + |G_i''|) = \sum_i \frac{1}{2}(|G_{i}'| + |G_{i+1}'|) + \sum_i |G_i''|.
\]
We apply Lemmas \ref{lemma:noncrossdist2} and \ref{lemma:dist1} to each term in the sums to obtain:
\[
3|\h| \leq  \sum_{i = 0}^{n-1} \frac{1}{2}n + \sum_{i = 0}^{n - 1} (n - 3) =  \frac{1}{2}n^2 + n(n - 3) = \frac{3}{2}n(n - 2).  \qed
\]

\section{Proof of Theorem \ref{thm:m3m2s3}: touching triangles with parallel sides, $S_2$}\label{sec:configs2}

Let $\h \subset \binom{\Omega_n}{3}$ be an $S_2$-free cgh. We are going to show $|H| \leq 23n^2/64$. Consider an edge $e = \{v_i,v_j,v_k\} \in \h$ where $v_i < v_j < v_k$. We call the pair $\{v_i,v_j\}$ {\em good for $e$} if there does not exists a $k'$ such that $v_j < v_{k'} < v_k$ and $\{v_i,v_j,v_{k'}\} \in \h$, and {\em bad} otherwise.

\begin{lemma}\label{good}
Let $\h \subset \binom{\Omega_n}{3}$ be an $S_2$-free cgh. Then \\
\textup{(1)} Every edge of $\h$ contains at least two good pairs. \\
\textup{(2)} Every pair in $\partial \h$ is good for either one or two edges of $\h$.
\end{lemma}

\begin{proof}
We first prove (1). Suppose $e = \{v_i,v_j,v_k\} \in \h$ and $\{v_i,v_j\}$ and $\{v_j,v_k\}$ are bad. Then there exist $k': v_j < v_{k'} < v_k$ and $i' : v_k < v_{i'} < v_i$ such that $\{v_i,v_j,v_{k'}\}, \{v_j,v_k,v_{i'}\} \in \h$. However, the edges $\{v_{i'},v_j,v_k\}$ and $\{v_i,v_j,v_{k'}\}$ form configuration $S_2$, a contradiction.

For (2), given $\{v_i,v_j\} \in \partial H$, consider an edge $\{v_i,v_j,v_k\}$ with $v_i < v_j < v_k$ and $v_k$ as close as possible to $v_j$; this determines $v_k$ uniquely. Similarly, for $\{v_i,v_j\} \in \partial H$, consider an edge $\{v_i,v_j,v_k\}$ with $v_i < v_k < v_j$ with $v_k$ as close as possible to $v_i$; this too determines $v_k$ uniquely. Therefore each pair in $\partial \h$ is good for either one of two edges of $\h$.
\end{proof}

Color a pair in $\partial \h$ {\em blue} if it is good for exactly one edge in $\h$, and {\em red} if it is good for exactly two edges in $\h$. Let $R$ be the number of red pairs and $B$ the number of blue pairs -- for a red pair $\{u,v\}$, there exist vertices $w,x \in \Omega_n$ on opposite sides of $\{u,v\}$ such that $\{u,v,w\} \in H$ and $\{u,v,x\} \in H$, so red pairs are what we have referred to as $D_1$-pairs in this paper.  If we map an edge $e \in \h$ to the pairs in $e$ that are good for $e$, then each red pair is counted twice and each blue pair is counted once. On the other hand, each edge of $\h$ contains at least two good pairs, by Lemma \ref{good}, so $2|H| \leq 2R + B$. In particular,
\[ |H| \leq R + B/2 \leq R + B = |\partial \h|.\]

\begin{lemma}\label{redtri}
If $\{v_i,v_j\}$, $\{v_j,v_k\}$ and $\{v_k,v_i\}$ are red pairs, then $\{v_i,v_j,v_k\} \in H$.
\end{lemma}

\begin{proof}
Suppose $\{v_i,v_j,v_k\} \not \in \h$ and $v_i < v_j < v_k$. Then by definition there exists $k' \neq k$ such that $\{v_i,v_j,v_{k'}\} \in \h$
and $v_j < v_{k'} < v_i$. We consider two cases.

\medskip

{\bf Case 1.} $v_j < v_{k'} < v_k$. There exists $i' \neq i$ such that $\{v_{i'},v_j,v_k\} \in \h$ and $v_k < v_{i'} < v_j$.
We observe $v_i < v_{i'} < v_j$, otherwise
$\{v_j,v_{k'}\}$ and $\{v_i,v_{i'}\}$ are non-crossing, and $\{v_i,v_j,v_{k'}\}$ and $\{v_{i'},v_j,v_k\}$ form $S_2$ in $\h$. Now there exists $j' \neq j$ such that $\{v_i,v_{j'},v_k\} \in \h$ and $v_i < v_{j'} < v_k$. If $v_i < v_{j'} < v_j$, then the pairs $\{v_{j'},v_k\}$ and $\{v_j,v_{k'}\}$ are non-crossing, and $\{v_i,v_j,v_{k'}\}$ and $\{v_i,v_j,v_{k'}\}$ form $S_2$. If $v_j < v_{j'} < v_k$, then $\{v_{i'},v_j\}$ and $\{v_i,v_j\}$ are ``parallel'', and $\{v_{i'},v_j,v_k\}$ and $\{v_i,v_j,v_{k'}\}$ form $S_2$ in $\h$.

\medskip

{\bf Case 2.} $v_k < v_{k'} < v_i$. Consider the reverse ordering of $\Omega_n$ and apply the proof of Case 1.
\end{proof}

By Lemma~\ref{redtri}, every triangle of red pairs is an edge of $\h$, so there are at most $|H| \leq |\partial H| \leq {n \choose 2}$ such triangles. In particular, the number of red pairs is at most $n^2/4 + n/2$ -- one could use a precise result by Lov\'{a}sz-Simonovits \cite{LS} to deduce this. Instead we give a direct proof:
the number of triangles in any graph $G$ is at least
\[ \sum_{\{u,v\} \in E(G)} (d(u) + d(v) - n).\]
If $G$ has average degree $d$, then this is precisely
\[ \sum_{u} d(u)^2 - \frac{1}{2}dn^2 \geq d^2 n - \frac{1}{2}dn^2.\]
Since the graph $G$ of red pairs in $\partial H$ has at most $|H| \leq {n \choose 2}$ triangles,
\[ d^2 n - \frac{1}{2}dn^2 \leq \frac{1}{2}n^2\]
which gives $d \leq n/2 + 1$ and therefore $R=|G| \leq n^2/4 + n/2$.
Therefore
\[ 2|H| \leq 2R + B \leq {n\choose 2} + (\frac{n^2}{4} + \frac{n}{2} )= \frac{3n^2}{4}. \]

To improve this bound to the desired $|H| \leq 23n^2/64$, we may assume $n$ is odd and partition the complete graph on $\Omega_n$ into planar matchings $M_1,M_2,\dots,M_n$ where $M_i = \{\{v_j,v_k\} : j + k \equiv i \mod n\}$. Then there exists $i \leq n$ such that at least $R/n$ pairs in $M = M_i$ are red. For each pair of red pairs, say $\{u,v\}$ and $\{w,x\}$, where $u < w < x < v < u$, there exist triples $\{u,v,y\},\{w,x,z\} \in H$ where $u < w < z < x < v < y < u$. Now by inspection, the pair $\{y,z\}$ cannot be contained in any edge of $H$ without creating configuration
$S_2$ -- see Figure \ref{proofs2.png}. Furthermore, if $\{u',v'\},\{w',x'\}\in M$, then $\{u',v',y\}$ and $\{w',x',z\}$ cannot both be edges of $H$ without creating $S_2$. Therefore for each pair $\{\{u,v\},\{w,x\}\}$ of red edges of $M$, we may associate a unique pair $\{y,z\}$ which is not contained in any edge of $H$.
Consequently
\[ 2|H| \leq 2R + B \leq 2R + {n \choose 2} - {R/n \choose 2} - R \leq R + {n \choose 2} - {R/n \choose 2}.\]
Since $R \leq n^2/4 + n/2$, this implies $|H| \leq 23n^2/64 - n/4 + 3/8$. As $n \geq 3$, this is at most $23n^2/64$, as required. \qed

 \begin{figure}[ht]
\centering
\includegraphics[scale=0.35] {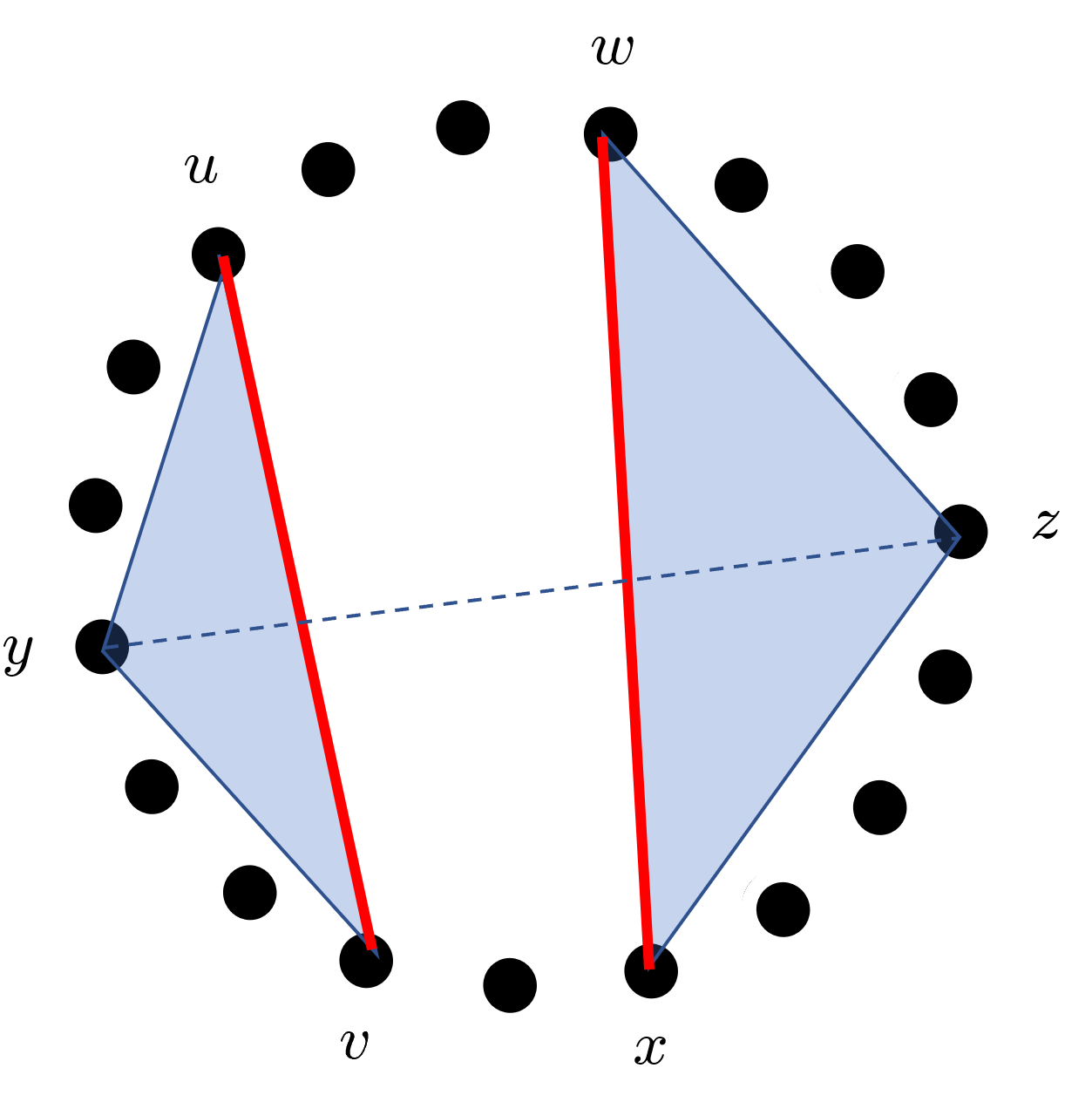}
\caption{Pair $\{y,z\}$ absent from $\partial H$}
\label{proofs2.png}
\end{figure}

\section{Proof of Theorem \ref{thm:configd2}: triangles sharing a side, $D_2$} \label{sec:configd2}

We first observe some simple bounds on $\cex(n,D_2)$. If $G$ is a convex geometric graph that is a triangulation of a convex polygon, then the family $T(G)$ of vertex sets of the triangular regions in $G$ form a $D_2$-free cgh.
By Euler's Formula, $|T(G)|< \frac{1}{2} |G|$, so if
 $G_1,G_2,\dots,G_M$ are edge-disjoint triangulations of polygons with vertices from $\Omega_n$, then $H = T(G_1) \cup T(G_2) \cup \dots \cup T(G_M)$ is a $D_2$-free cgh on $\Omega_n$.
Each  $D_2$-free cgh $H$ can be obtained in this way, so we get $\cex(n, D_2) < (1/2){n \choose 2}$.
 On the other hand, every Steiner triple system induces a  $D_2$-free cgh, we get
 $ \cex(n,D_2) \ge  \frac{1}{3} {n \choose 2} - O(n)$. Construction~8 improves this to $\frac{3}{7}{n \choose 2} - O(n)$,
 and Dam\'asdi and N.~Frankl~\cite{DF} showed $\cex(n,D_2) \geq \frac{2n^2-3n}{9}$ for all $n \equiv 6 \mod 9$
 by a different method.
Here we prove the upper bound $\cex(n,D_2) \leq \frac{2n^2-3n}{9}$ for all $n$.

For the calculation below we need a simple proposition
which can be shown by standard high school calculus.
If $h, x\geq 0$ are reals, $n\geq 3$ is an integer and $h\geq (2n-3)/9$,  then
\begin{equation}\label{eq101}
(h+2x)(h+2x+1)\leq 2xn \quad \Longrightarrow \quad x\geq \frac{n+3}{18}.
  \end{equation}

Another elementary proposition is the following statement:
Suppose that $A$ is a {\em multiset} of positive integers such that the multiplicity of each entry is at most $n$, then
\begin{equation}\label{eq102}   \sum_{a\in A} a \geq \frac{|A|(|A|+n)}{2n}.
  \end{equation}

 \medskip
For the upper bound on $\cex(n, D_2)$, let $\h \subset \binom{\Omega_n}{3}$ be a $D_2$-free cgh.
The graph $\partial H$ has a (unique) edge-disjoint decomposition into triangulations $G_1, \dots, G_M$ as follows.
Make a graph $C$ with vertex set $H$: two triangles of $H$ are joined by an edge of $C$ if they share a side.
Consider the partition of $C$ generated by the components $C_1, C_2, \ldots C_M$ of $G$, where $|C_i|=k_i$.
Each $C_i$ corresponds to a hypergraph $H_i\subset H$ of triangles. Since $H_i$ is $D_2$-free each $G_i:=\partial H_i$ forms a triangulation of a convex  $(k_i+2)$-gon $P_i$ with $k_i-1$  diagonals, $T(G_i)=H_i$, $|E(G_i)|=2k_i+1$.
Let $A_i$ be the multiset of integers  consisting of the side lengths of $P_i$,
$|A_i|=k_i+2$.  We have
\begin{equation}\label{eq103} \sum_{a\in A_i} a \leq n\end{equation}
and here equality holds if the polygon $P_i$ contains the center of $\Omega_n$.
Let $A$ be the multiset $\cup_{i\leq M} A_i$. 
Since each edge of $\partial H$ appears in exactly one $G_i$ and there are $n$ (or $n/2$ or $0$)
 diagonals of $\Omega_n$ of a given length we obtain that $A$ is a multiset with maximum multiplicities at most $n$.
Moreover, $|A|= \sum_i (k_i+2)=|H|+2M$, so~\eqref{eq102} and~\eqref{eq103} yield
\begin{equation}\label{eq104}  \frac{(|H|+2M)(|H|+2M+n)}{2n} \leq \sum_{a\in A} a =\sum_{i\leq M} (\sum_{a\in A_i} a)\leq Mn.
  \end{equation}
Suppose that $|H|\geq (2n^2-3n)/9$. Define $h,x$ as $h:=|H|/n$ and $x:=M/ n $.
Then $h\geq (2n-3)/9$ and~\eqref{eq104} and~\eqref{eq101} imply $x\geq  (n+3)/18$.
However
\[ 2|H|+M =  \sum_{1\leq i\leq M} (2k_i+1) = \sum |E(G_i)|=|\partial H|\leq {n\choose 2}.
  \]
Hence $|H|\leq \frac{1}{2}({n\choose 2}-xn)\leq  (2n^2-3n)/9$. \qed

\section{Concluding Remarks}\label{sec:remarks}

$\bullet$ In this paper, we considered  convex geometric configurations consisting of two triples. One may consider analogous problems
for $r$-tuples: for instance, how many edges can a convex geometric $n$-vertex $r$-graph have if it does not contain two hyperedges which are geometrically
disjoint as $r$-gons (this is the $r$-uniform analog of $M_1$)? This problem was posed explicitly by P.~Frankl, Holmsen and Kupavskii~\cite{FHK}:

\begin{problem}
Find analogues of our results for other classes of sets such as convex $r$-gons in $\mathbb R^2$.
\end{problem}

A family of convex $r$-gons in the plane is {\em strongly intersecting} if any two of the members share a point in their interior.
The maximum size of a strongly intersecting family of $r$-gons is obtained from the obvious extensions of Construction~\ref{1}.
Consider the family of all $r$-gons containing the centroid of $\Omega_n$ when $n$ is odd, together with, for each diameter $\ell$, all $r$-gons which have a side equal to $\ell$ and which lie on one side of $\ell$.  Letting $\dotDelta_r(n)$ denote the size of these families, it is not hard to see
\[ \dotDelta_r(n) = {n \choose r} - n{(n - 1)/2 \choose r - 1}\]
if $n$ is odd, and $\dotDelta_r(n)$ can be computed similarly if $n$ is even.
In particular, $\dotDelta_r(n) = (1 - r/2^{r-1}){n \choose r} + O(n^{r-1})$ for each $r \geq 3$.

\begin{thm}\label{th:13}
The maximum size of a strongly intersecting family of $r$-gons from $\Omega_n$ is $\dotDelta_r(n)$.
\end{thm}

\begin{proof} (Sketch).
We proceed in a similar way to the proof of Theorem \ref{thm:all} for $M_1$.
Consider any $r$-gon $\{v_{i_1},v_{i_2},\dots,v_{i_r}\}$ in $H$
with $v_{i_1} < v_{i_2} < \dots < v_{i_r} < v_{i_1}$ and where the longest side $\{v_{i_1},v_{i_r}\}$ is as short as possible, and replace all such $r$-gons
with $\{v_{i_1},v_{j_2},\dots,v_{j_{r-1}},v_{i_r}\}$ where $v_{i_r} < v_{j_2} < v_{j_3} < \dots < v_{j_{r-1}} < v_{i_1}$. Since the number of choices
of $j_2,j_3,\dots,j_{r-1}$ is always at least the number of choices of $i_2,i_3,\dots,i_{r-1}$, this new $r$-cgh $H'$ has $|H'| \geq |H|$. So we repeat
until $H'$ consists of all $r$-gons containing the centroid of $\Omega_n$ when $n$ is odd, or $n$ is even and $H'$ consists of all $r$-gons containing
the centroid plus for each diameter $\ell$ all $r$-gons which have a side equal to $\ell$ and which lie on one side of $\ell$.
\end{proof}

\medskip

$\bullet$ Since there are many other possible configurations of two $r$-gons, or two ordered $r$-tuples, we did not discuss
these problems in this paper. Some special cases were studied in~\cite{FJKMV1}: for instance, if $F$ consists of two $r$-tuples $\{u_1,u_2,\dots,u_r\}$ and $\{v_1,v_2,\dots,v_r\}$
where $u_1 < v_1 < u_2 < v_2 < \dots < u_r < v_r < u_1$, then it was shown in~\cite{FJKMV1} that for $n > r > 1$,
\[ \cex(n,F) = {n \choose r} - {n-r \choose r}.\]
This may be viewed as a geometric or ordered version of the Erd\H{o}s-Ko-Rado Theorem~\cite{EKR}.

\medskip

$\bullet$ In the cases of $M_2, M_3$ and $S_3$ (see Figure \ref{allconfig}), we obtained exact results for the extremal functions in convex geometric hypergraphs / convex triangle systems (for $n$ even in the case of $S_3$). Our proofs, with more work, should give a characterization of the extremal examples as well. For $M_2$, one requires $n \geq 8$ for the extremal configuration to be unique, as verified by computer. For $S_3$, we believe that $\cex(n,S_3) = (n - 1)(n - 2)/2 + 1$ when $n$ is odd, but do not have a proof,
and we also do not know the characterization of extremal $S_3$-free convex triangle systems (this is the content of Problem~\ref{prob:s3}).

\medskip

$\bullet$
It is likely the case that most of our theorems hold equally for {\em ordered hypergraphs}, where the vertex set is linearly ordered,
but we did not work out the details except for the obvious case $M_3$ (see the first paragraph in Section~\ref{sec:proofofthmm2m3}). The case
of $S_2$ stands out, since the ordered extremal number is not the same as the convex geometric extremal number. The ordered construction
would be to take all triples $\{v_i,v_{i + 1},v_j\}$ from an ordered vertex set $\{v_0,v_1,\dots,v_{n-1}\}$ where $i \geq 0$ and $i + 1 < j \leq n - 1$.

Extremal problems for matchings in ordered graphs connect to enumeration of permutations~\cite{MT} and these have also been extended to hypergraphs~\cite{KM}.

\medskip

$\bullet$ A hypergraph $H$ is {\em linear} if for distinct hyperedges $e,f \in E(H)$, $|e \cap f| \leq 1$.
The extremal functions for the configurations in this paper in the context of linear cghs were determined in~\cite{ADMOS} up to constant
factors for all the configurations except $S_2$. Specifically, if $\cex^*(n,F)$ is the maximum number of triples in an $n$-vertex $F$-free
linear cgh, then Aronov, Dujmovi\'c, Morin, Ooms and da Silveira~\cite{ADMOS} proved $\cex^*(n,M_2) = \Theta(n)$, whereas
if $F \in \{M_1,M_3,S_1,S_3\}$, $\cex^*(n,F) = \Theta(n^2)$. It would be interesting to determine the exact extremal functions in each case.
The problem of determining $\cex^*(n,S_2)$ appears to be very difficult, as it is connected
to monotone matrices, tripod packing, and 2-comparable sets  -- see Aronov, Dujmovi\'c, Morin, Ooms and da Silveira~\cite{ADMOS} for details.
The best bounds are $\cex^*(n,S_2) = \Omega(n^{1.546})$ due to Gowers and Long~\cite{GL} and $\cex^*(n,S_2) = n^2/\exp(\Omega(\log^*n))$ due to the best bounds on
the removal lemma by Fox~\cite{Fox}.

\medskip

$\bullet$
By a result of  Boros and F\"uredi~\cite{ZF33}, for every $n$-point set $P$ (no three on a line) one can find a point on the plane which is contained in at least $n^3/27 -O(n^2)$ triangles with these vertices; and Bukh,  Matou\v{s}ek, and Nivasch~\cite{BMN} gave an example that the coefficient $1/27$ is the best possible. It would be interesting to determine the largest subsystem of pairwise intersecting triangles in this construction.

\medskip

$\bullet$
One can further relax the conditions on the point sets to allow {\em all} planar $n$-point sets.  We conjecture that our upper bounds
in Theorem~\ref{thm:allplanar}
hold for all planar $n$-point sets (when we only count the proper triangles with non-empty interiors). Surely in that case one has to relax the definition of configurations (like, e.g., Ackerman, Nitzan, and Pinchasi~\cite{ANP} did about avoiding pairs of edges).

\medskip

$\bullet$
We have not considered $F$-free  triangle systems $(P,\TT)$ where the point set $P$ is not necessarily in convex position and    $F \in \{M_2,M_3,S_2,S_3,D_2\}$.
The reason is, unlike in the case $F\in \{ D_1, S_1, M_1\}$, there are many different ways to extend the definitions of these configurations and these can lead to many different problems.
E.g., if one insists that no triangle in $F$ contains another vertex of $F$ then the answer is always at least $n^3/27 +O(n^2)$ as it is shown by the following example
$P:=X\cup Y \cup Z$, $\TT:=\{ xyz: x\in X, y\in Y, z\in Z \}$ and $X:=\{ (i, 10^{-i}): 1\leq i\leq n/3\}$, $Y:=\{ (10^{-i}, i): 1\leq i\leq n/3\}$, and
$Z:=\{ (-i, -i+10^{-i}): 1\leq i\leq n/3\}$.
It is a rich area with full of problems, e.g., it would be interesting to determine all configurations $F$ satisfying that $|\TT| \leq (1+(o(1)) \cex(n,F)$
holds for $F$-free triangle systems.



\begin{thebibliography}{99}
{\small



\bibitem{ANP}
E. Ackerman, N. Nitzan, R. Pinchasi, The maximum number of edges in geometric graphs with pairwise virtually avoiding edges, Graphs Combin. 30, 1065--1072, 2014.


\bibitem{ADMOS} B. Aranov, V. Dujmuvi{\' c}, P. Morin, A. Ooms, L. Xavier da Silveira, More Tur\'{a}n-type theorems for triangles in convex point sets, Electronic Journal of Combinatorics, 26, 2019.

 \bibitem{Barany} I. B\'{a}r\'{a}ny, A generalization of Carath\'{e}odory’s theorem, Discrete Math.,  40,
141--152, 1982.


\bibitem{ZF1} E. Boros, Z. F\"uredi,
  Su un teorema di K\'arteszi nella geometria combinatoria, (Italian),
   Archimede, 29,  71--76,  1977.

\bibitem{ZF33} E. Boros, Z. F\"uredi,
The number of triangles covering the center of an $n$-set,
  Geometriae Dedicata, 17,  69--77, 1984.

\bibitem{Brass1} P. Bra{\ss}, Tur\'{a}n-type extremal problems for cghs, Towards a theory of geometric graphs, 25--33, Contemp. Math., 342, Amer. Math. Soc., Providence, RI, 2004.

\bibitem{Brass2} P. Bra{\ss}, G. Rote, and K. J. Swanepoel, Triangles of extremal area or perimeter in a finite planar point set, Discrete $\&$ Computational Geometry, 26, 51--58, 2001.

\bibitem{Bukh} B. Bukh, A point in many triangles, Electron. J. Comb., 13, 2006.

\bibitem{BMN}
B, Bukh, J. Matou\v{s}ek, G. Nivasch, Stabbing simplices by points and flats,
Discrete $\&$ Computational Geometry, 43, 321--338, 2010.

\bibitem{CP} V. Capoyleas, J. Pach, A Tur\'{a}n-type theorem for chords of a convex polygon, J. Combin. Theory Ser. B, 56, 9--15, 1992.

\bibitem{DF}  G. Dam\'asdi, N. Frankl, $D_2$-free cghs, manuscript, December 2020.


\bibitem{EKR} P. Erd\H{o}s, C. Ko, R. Rado, Intersection theorems for systems of finite sets, The
Q. J. Math., 12, 313--320, 1961.

\bibitem{ESp} P. Erd\H{o}s, J. Spencer, Probabilistic methods in combinatorics,
Probability and Mathematical Statistics, Vol. 17. Academic Press [A
subsidiary of Harcourt Brace Jovanovich, Publishers], New York-London,  106 pp., 1974.



\bibitem{Fox} J. Fox, A new proof of the graph removal lemma, Annals of Mathematics, 174, 561--579, 2011.

\bibitem{FHK} P. Frankl, A. Holmsen, A. Kupavskii, Intersection theorems for triangles,\newline  
{\tt \href{ https://urldefense.com/v3/__https://arxiv.org/abs/2009.14560__;!!Mih3wA!T2NXWrE3k4K9JRmQnrUQT2MFmyiJMCuFmAyJK4FioVWdTI5cwpGSRC-yMezjpbTd$ }{https://arxiv.org/abs/2009.14560}}, 2020.

\bibitem{FJKMV1} Z. F{\" u}redi, T. Jiang, A. Kostochka, D. Mubayi, J. Verstraete, Extremal problems on ordered and convex geometric hypergraphs, Canadian Journal of Mathematics, 1--21, 2020.

\bibitem{FJKMV2} Z. F{\" u}redi, T. Jiang, A. Kostochka, D. Mubayi, J. Verstraete, Tight paths in convex geometric hypergraphs, Advances in Combinatorics, 2020.

\bibitem{FKMV1} Z. F{\" u}redi, T. Jiang, A. Kostochka, D. Mubayi, J. Verstraete, Ordered and convex geometric trees with linear extremal function,  Discrete $\&$ Computational Geometry, 64, 324--338, 2020.

\bibitem{GL} W.  T.  Gowers  and  J.  Long,   The  length  of  an $s$-increasing  sequence  of $r$-tuples, \newline
{\tt \href{https://urldefense.com/v3/__https://arxiv.org/abs/1609.08688__;!!Mih3wA!T2NXWrE3k4K9JRmQnrUQT2MFmyiJMCuFmAyJK4FioVWdTI5cwpGSRC-yMf7X3g1G$}{https://arxiv.org/abs/1609.08688}}, 2016.

\bibitem{Gromov} M. Gromov, Singularities, expanders and topology of maps. Part 2: from combinatorics to topology via algebraic isoperimetry, Geom. Funct. Anal., 20, 416--526, 2010.


\bibitem{HP} H. Hopf and E. Pannwitz, Aufgabe Nr. 167, Jahresbericht. Deutsch. Math.-Verein., 43, 114, 1934.

\bibitem{Karasev} R. Karasev, A simpler proof of the Boros-F\"{u}redi-B\'{a}r\'{a}ny-Pach-Gromov theorem, Discrete Comput. Geom., 47, 492--495, 2012.

\bibitem{KeP} C. Keller, M. Perles, On convex geometric graphs with no $k +1$ pairwise disjoint edges, Graphs Combin., 32, 2497--2514, 2016.


\bibitem{KM} M. Klazar, A. Marcus, Extensions of the linear bound in the F\"{u}redi-Hajnal conjecture,
Adv. in Appl. Math. 38, 258--266, 2007.


\bibitem{Kup} Y. S. Kupitz, On Pairs of disjoint segments in convex position in the plane, Annals Discrete Math., 20, 203--208, 1984.

\bibitem{KuP}  Y. S. Kupitz, M. Perles, Extremal theory for convex matchings in convex geometric graphs, Discrete $\&$ Computational Geometry, 15, 195--220, 1996.

\bibitem{LS} L. Lov\'{a}sz, M. Simonovits, On the number of complete subgraphs of a graph II., Studies in Pure Mathematics, 459--495, 1983.

\bibitem{MT}  A. Marcus, G. Tardos, Excluded permutation matrices and the Stanley-Wilf conjecture, Journal
of Combinatorial Theory, Ser. A, 107, 153--160, 2004.

\bibitem{Moon} J. W. Moon, Topics on tournaments,  Holt, Rinehart and Winston, New
York-Montreal, Que.-London, viii+104 pp., 1968.


\bibitem{P1} J. Pach, Geometric graph theory, Surveys in combinatorics, 1999 (Canterbury), 167--200, London Math. Soc. Lecture Note Ser., 267, Cambridge Univ. Press, Cambridge, 1999.

\bibitem{P2} J. Pach, The beginnings of geometric graph theory, Erd\H{o}s centennial, 465--484, Bolyai Soc. Math. Stud., 25, János Bolyai Math. Soc., Budapest, 2013.

\bibitem{PP}  J. Pach, R. Pinchasi, How many unit equilateral triangles can be generated by $n$ points in general position? Amer. Math. Monthly, 110, 100--106, 2003.

\bibitem{PPTT} J. Pach, R. Pinchasi, G. Tardos, G. T\'oth,  Geometric graphs with no
self-intersecting path of length three,  European J. Combin., 25, 793--811, 2004.


\bibitem{PT} J. Pach, G. Tardos, Forbidden paths and cycles in ordered graphs and
matrices, Israel Journal of Mathematics 155, 359--380, 2006.


\bibitem{SAGE} SageMath, the Sage Mathematics Software System (Version 9.1), The Sage Developers, \newline
{\tt \href{https://urldefense.com/v3/__https://www.sagemath.org__;!!Mih3wA!T2NXWrE3k4K9JRmQnrUQT2MFmyiJMCuFmAyJK4FioVWdTI5cwpGSRC-yMcLfpqVK$}{https://www.sagemath.org/}}, 2020.

\bibitem{S} J. W. Sutherland, L\"{o}sung der Aufgabe 167, Jahresbericht Deutsch. Math.-Verein., 45,
33--35, 1935.


\bibitem{T} G. Tardos, Extremal theory of ordered graphs, Proceedings of the International Congress of Mathematics -- 2018, Vol. 3, 3219--3228.


\bibitem{Valtr} P. Valtr, On geometric graphs with no $k$ pairwise parallel edges, Discrete $\&$ Computational Geometry, 19), 461--469, 1998.

\bibitem{WW} U. Wagner, E. Welzl, A continuous analogue of the upper bound theorem,
Discrete Comput. Geom., 26, 205--219, 2001.


\bibitem{W} R. M. Wilson,  An existence theory for pairwise balanced designs, I and II,
J. Combinatorial Theory Ser. A, 13 , 220--245 and 246--273, 1972.

}
\end{thebibliography}
\end{document}